\documentclass[11pt]{article}

\usepackage{layout, graphicx, amsmath, amsthm, amssymb, euscript, enumerate, bbold, bbm, graphicx, fancyhdr,xcolor,mathtools, tikz,url,mathrsfs}
\usepackage{float}
\usepackage{subcaption}
\usetikzlibrary{arrows,shapes}
\usepackage{verbatim}
\definecolor{lila}{RGB}{150,51,255}
\definecolor{verdesi}{RGB}{0,153,0}
\usepackage{yfonts} 
\usetikzlibrary{calc,trees,positioning,arrows,chains,shapes.geometric,%
	decorations.pathreplacing,decorations.pathmorphing,shapes,%
	matrix,shapes.symbols,plotmarks,decorations.markings,shadows}
\usepackage{caption}
\usepackage{subcaption}

\usepackage[pdftex,bookmarks,colorlinks]{hyperref}
\hypersetup{colorlinks=false}

\makeatletter                                         
\def\th@newremark{\th@remark\thm@headfont{\bfseries}}  
\makeatletter                                          

\theoremstyle{newremark}                               

\newtheorem{theorem}{Theorem}[section]
\newtheorem{corollary}[theorem]{Corollary}

\newtheorem{lemma}[theorem]{Lemma}
\newtheorem{proposition}[theorem]{Proposition}
\newtheorem{remark}[theorem]{Remark}

\newtheorem{notation}{Notation}
\DeclareMathOperator*{\esssup}{ess\,sup}
\DeclareMathOperator*{\argmin}{arg\,min}

\definecolor{burgundy}{rgb}{0.5, 0.0, 0.13}
\definecolor{linkblue}{RGB}{0,20,128}
\definecolor{linkred}{RGB}{128, 0, 6}
\definecolor{citegreen}{RGB}{46, 126, 42}
\definecolor{urlmagenta}{RGB}{138, 0, 135}
\hypersetup{
	linkcolor  = linkblue,
	citecolor  = citegreen,
	urlcolor   = urlmagenta,
	colorlinks = true,
}
\makeindex
\begin{document}
	\title{Microscopic Weak Selection Principle for the Logistic Branching Brownian Motion with selection}
	
	\author{%
		F.~E. Bravo Lozano\thanks{Facultad de Ciencias,  
			UNAM, M\'exico. Email: {frank.bravo@ciencias.unam.mx}}
		\and 
		M.~C. Fittipaldi\thanks{Facultad de Ciencias, UNAM, Mexico. Email: {mcfittipaldi@ciencias.unam.mx} }}
\maketitle		
		\begin{abstract}
			In this work, we present the \emph{logistic branching Brownian motion with selection} (Log-BBM), a modification of the $N$-BBM defined by Groisman et. al (2020), in which birth and competition events are decoupled to allow for a variable population size that follows the branching process with logistic growth defined by Lambert (2005). We study the representation of the Log-BBM as a microscopic model for a FKPP-type equation. In the large population limit, the renormalised empirical measure of the Log-BBM converges weakly to a probability measure whose density solves a nonlocal version of the FKPP equation, while its  cumulative distribution function solves the classical FKPP equation. We also show that this model exhibits behaviour that is characteristic of the Brunet–Derrida family of systems, in particular the so-called \emph{weak selection principle}. Indeed, we show that, in the Log-BBM, the particles select the minimal propagation speed, in contrast to the FKPP equation, where the selected front speed depends on the initial condition.
		\end{abstract}


		
\section{Introduction}
The study of branching processes, a term originally coined by A. N. Kolmogorov and N. A. Dimitriev \cite{KolmogorovDmitriev1947} in 1947 to describe stochastic processes arising when probability theory is applied to population mathematics, originated from the work of Bienaym\'e \cite{Bienayme1845}, and Galton and Watson\cite{GaltonWatson1874}. However, it was not until 1958 that Sevast’yanov \cite{Sevastyanov1958} considered a population in which members diffuse according to a Gaussian process within a region bounded by an absorbing boundary. In the early 1960s,  Adke and Moyal \cite{Adke1963}, seeking to describe both growth and dispersal of a population in an unrestricted linear habitat, removed the boundary constraint and instead introduced deaths at a constant rate. Skorokhod \cite{Skorokhod1964} further generalized this framework by allowing a more general diffusion for the particles. While all of these earlier processes can be regarded as particular instances of what is now called {\it branching Brownian motion (BBM)}, the term itself appeared for the first time in \cite{Ikeda1968a}.

The dynamics of branching Brownian motion can be described as follows. A single particle starts at position 
$x \in \mathbb{R}$ at time $t=0$ and evolves according to a standard Brownian motion. The particle carries an exponential clock with rate 
$\beta > 0$. When the clock rings, the particle reproduces into 
$k$ particles (the offspring) with probability 
$p(k)$. Each offspring then begins an independent standard Brownian motion from its birth site equipped with its own exponentially distributed clock with the same rate $\beta$, and the process continues recursively.

The discrete (in space and time) analogue of BBM is the \emph{branching random walk (BRW)}. In this model, the population size evolves according to a discrete-time Galton–Watson process, and 
particles move independently on the lattice according to a prescribed random walk. The study of both models has developed significantly and both serve as examples of tree-indexed Markov processes. Interest in BBM grew considerably after McKean \cite{McKean1975} established a fundamental connection with a reaction–diffusion equation in the mid-1970s. Beyond ongoing research related to McKean’s link, the process has since been extensively studied as a probabilistic object.


A natural modification of the original BBM is to introduce interactions between particles. In this direction, Brunet and Derrida \cite{Brunet1997,Brunet1999,Brunet2001} proposed a discrete-time model with fixed population size $N$. The dynamics are as follows: consider $N$ particles moving on the line, with positions $Y_1(t), \ldots, Y_N(t)$. For each particle $i \in [1,N]$, two indices $j_i$ and $j'_i$ are chosen at random from among the $N$ particles (possibly including $i$ itself). The position of particle $i$ at time $t+1$ is then given by
\begin{equation*}
	Y_i(t+1) = \max\left[Y_{j_i}(t)+\alpha_i,Y_{j'_i}(t)+\alpha_i'\right],  
\end{equation*}
where $\alpha_i$ and $\alpha'_i$ are independent Bernoulli random variables with parameter $p$. 
For this model, Brunet and Derrida conjectured that as $N$ goes to infinity, the function $h(t,x)$ which represents the fraction of particles to the right of $x$ at time $t$, evolves according to a reaction–diffusion equation. They further used this observation, empirically, to investigate the microscopic fluctuations of the moving wavefront, which led to a series of conjectures about the behavior of the model. In particular, these conjectures are expected to hold for a broad class of processes now referred to as \emph{Brunet–Derrida particle systems}.
In a later work, these authors introduced a number of models (some previously known and others newly proposed) that belong to this family (see \cite{Brunet2006}). One example is the \emph{coalescing BBM}, in which particles, in addition to branching, coalesce at a small rate $\varepsilon^2$ during their intersection local time. Two further examples are the \emph{$N$-BBM}, first rigorously defined and studied in \cite{Maillard2016}, and the \emph{$L$-BBM}, first studied in \cite{Pain2016}. In both models, particles evolve as in standard BBM but they are also subject to an additional selection mechanism. In the former model, the leftmost particle is removed at every branching event. In the latter model, particles are removed whenever they fall more than a distance $L$ behind the rightmost particle.

Groisman et. al \cite{Groisman2020} defined a continuous generalization of the Brunet-Derrida model that can be viewed as a modification of the $N$-BBM.  In their version, the population size remains fixed according to the following dynamics: at rate one, each particle randomly selects another particle, and the particle in the lower position jumps over the other particle. 
Later, Groisman and Soprano-Loto \cite{Groisman2021} generalized this model by considering more general interaction rates and jump mechanisms. This allows for a broader class of fixed-size particle systems with selection dynamics.

In the above‑described setting, it is natural to ask what happens when branching and competition do not occur simultaneously and, consequently, the population size is no longer fixed. In this case, particles would branch, and each pair would compete at a logistic rate, meaning there would be a quadratic rate of deaths in the system. This underlying branching structure resembles the one proposed by Lambert \cite{Lambert2005} for the \emph{branching process with logistic growth} (LB-process). 
In the present work we study a modification of the $N$-BBM with the following dynamics. We start with an initial number of particles, the count of which follows the stationary distribution given in \cite{Lambert2005} in the absence of natural death. Each particle then moves independently as a Brownian motion. At rate one, each particle reproduces, and its offspring start from the position of their parent and continue to move independently as Brownian motions. In addition, at rate $c$, each particle selects another particle uniformly at random, and the particle with the leftmost position is removed from the system. We will call this process the {\bf Logistic Branching Brownian Motion with selection} ({\bf Log-BBM} for short). The formal definition of this process and its construction using a coupling with the BBM are given in Section \ref{Log}.  

\subsection{Probabilistic representations of KPP-type equations via branching particle systems}\label{nexus}
In particular, we are interested in describing the connection between the Log-BBM and KPP-type equations. Fisher
\cite{Fisher1937} introduced the so-called FKPP equation  while describing the dynamic of a population in which a mutation occurs. 
In this scenario, individuals with the new trait have a survival advantage. In the same year Kolmogorov et. al \cite{Kolmogorov1937} arrived at a variation of the FKPP equation under similar conditions. 
The authors gave a detailed analysis of the existence of traveling wave solutions for a general version of the equation studied by Fisher, given by
\begin{equation}\label{eq: fkpp}
	\partial_t u = k\partial_{xx}u + H(u),
\end{equation}
whose forcing term $H$ is assumed to be in $C^1([0,1])$ and satisfy $H(0)=H(1)=0, H(u)> 0$ for $0<u<1$, and  $H'(0)=1, H'(u)\leq 1$ for $0<u\leq 1$. In particular, the FKPP equation is a semilinear equation, where diffusion is coupled with logistic growth throughout the reaction term, namely $H(u) = m u (1-u)$. In this probabilistic setting, it is a convention to take $k = 1/2$ and consider the transformation $u \to 1 -u$. 

The KPP family of equations emerges as the macroscopic limit of the underlying microscopic dynamics or through duality relations in a wide variety of branching particle systems, which is why there is so much interest in the probability literature.
The first connection between the branching Brownian motion and the Kolmogorov equation is attributed to McKean \cite{McKean1975}. 
Consider $\mathbf{X}(t) = (X_u(t), u\in\mathcal{N}_t)$ the original BBM with reproduction law $(p(k))_{k\geq 0}$ and branching rate $\beta$. McKean proved that the solutions of the Kolmogorov equation with reaction term $H(u) = \beta(f(u)-u)$ (where $f(u)$ is the generating function of the offspring distribution) and initial condition $g(x)$ can be written as
\[
u(t,x) = \mathbb{E}\bigg[\prod_{v\in\mathcal{N}_t}g(x - X_v(t))\bigg].
\]
In particular, in the \emph{dyadic case}, when $p(2)=1$, the solution when $g(x) = \mathbb{1}_{\{x>0\}}$ is exactly the cumulative distribution function of the position of the rightmost particle, namely 
\begin{align*}
	u(t,x) = \mathbb{P}\Big(\max_{v\in\mathcal{N}_t}X_v(t)<x\Big).
\end{align*} 
On the other hand, if we consider the BBM starting with $N$ particles and take its (normalized) empirical measure, i.e. 
\begin{align*}
	\mu_t^N:=\frac{1}{N}\sum\limits_{v\in\mathcal{N}_t}\delta_{X_u(t)},    
\end{align*}
then, as the initial number of particles goes to infinity, the measure $(\mu_t^N)_{t\geq 0}$ converges weakly to a measure $(\mu_t)_{t\geq 0}$ whose density $u(t,\cdot)$ solves the linear Kolmogorov equation (see \cite{Beckman2019}). In the same fashion, for the $N$-BBM, the selection dynamics generate a free boundary at the position of the leftmost particle, which is described by a stochastic process $(L_t)_{t\geq 0}$. In this case, the pair $(u(t,\cdot),L_t)$ satisfies a free‑boundary problem in which the governing equation is the linear Kolmogorov equation (see \cite{DeMasi2019}). 

Meanwhile, for the process defined in \cite{Groisman2020}, the authors showed that the empirical distribution of the system (i.e. $\mu_t^N((-\infty,x])$) converges to the solution of the FKPP equation. 
For the generalizations proposed in \cite{Groisman2021}, they proved a similar result; however, in these cases, the reaction term is determined by the specific jump mechanisms introduced in the model.

In line with some of the cited works, in Section \ref{hydrodynamic} we prove a hydrodynamic limit for our model. Specifically, we show that, in the large‑population limit, the empirical measure converges weakly to a measure whose density function satisfies a nonlocal version of the FKPP equation. Before stating our result, we introduce the necessary notation.

Let $K\in\mathbb{N}$ be a scale parameter to the competition rate. We set $c_K:=cK^{-1}$, which tends to zero when $K$ tends to infinity. 
For each $K\in\mathbb{N}$, let $\mathbf{X}^K=(\mathbf{X}^K(t))_{t\geq 0}$ be the (rescaled) Log-BBM($1, c_K$), given by $\mathbf{X}^K(t)=\{X^K_u(t), u \in \mathcal{N}^K(t)\}$, where the  initial number of particles is given by the stationary distribution of the LB-process.  Let 
\begin{equation}\label{eq: defemp0}
	\mu^K_t (\cdot) := \frac{1}{m_K} \sum\limits_{u \in \mathcal{N}^K_t} \delta_{X^K_u(t)}(\cdot), \qquad t\geq 0,
\end{equation}
be the empirical measure of the system $\boldsymbol{X}^K$, where $m_K$ is the mean of the stationary distribution in the number of particles.

\begin{notation}\label{not}
	In what follows, we denote by $M_F(\mathbb{R})(\text{resp. }M^v_F(\mathbb{R})$) the space of positive and finite measures on $\mathbb{R}$ endowed with the weak (resp. vague) topology. We denote by $D\big([0,T], S\big)$ the space of \emph{c\`adl\`ag} functions with values in $S$, which can be $M_F(\mathbb{R})$,  $M^v_F(\mathbb{R})$ or $\mathbb{R}$. Likewise, we denote by $C^{1,2}_b(\mathbb{R}_+\times \mathbb{R})$ (resp.  $C^{1,2}_0(\mathbb{R}_+\times\mathbb{R})$, $C^{1,2}_c(\mathbb{R}^+\times\mathbb{R})$) the space of bounded (resp. vanishing at infinity, of compactly supported) functions which are continuously differentiable up to the order $2$ in the space variable and continuously differentiable in the time variables with bounded partial derivatives.    
\end{notation}

We can now establish the following result, whose proof can be found in Section \ref{hydrodynamic}.
\begin{theorem}\label{thm: convemp}
	For each $T>0$, the sequence of probability measures $\{\mu^K\}_{K\in\mathbb{N}}$ defined by \eqref{eq: defemp0} converges in law in $D\left([0,T], M_F(\mathbb{R})\right)$ to a continuous deterministic function $\mu:=(\mu_t)_{t\in[0,T]}$. In particular, the measure-valued function $\mu$ is the unique solution of the following PDE, written in its weak form
	\begin{equation}\label{eq: eqmu0}
		\left<\phi, \mu_t\right> = \left<\phi, \mu_0\right> + \int_0^t\int_{\mathbb{R}}\Bigg[\partial_t \phi(s,x) + \frac{1}{2}\partial_{xx} \phi(s,x)
		+ \phi(s,x) \left(1- 2 \int_x^{\infty} \mu_s(dy)\right)\Bigg]\mu_s(dx)ds,
	\end{equation}
	for all bounded functions $\phi(t,x) \in C^{1,2}_b(\mathbb{R}_+\times\mathbb{R})$.
	
	Moreover, if $\mu_0$ has a density with respect to the Lebesgue measure, then for every $t \in [0,T]$ the finite measure $\mu_t$ has a density $u$ with respect to the Lebesgue measure, such that $\mu_t(dx) = u(t,x)dx$, and $u$ is a weak solution of the (nonlocal and nonlinear) partial integro-differential equation
	\begin{equation}\label{eq: densitymu0}
		\begin{split}
			\partial_t u(t,x)=  \frac{1}{2}\partial_{xx} u(t,x)
			+ u(t,x)\Big(1 - 2  \int_x^{\infty} u(t,y)dy\Big)\
		\end{split} .
	\end{equation}
\end{theorem}

\begin{remark}
	Note that equation \eqref{eq: densitymu0} can be seen as a non local FKPP. There has been a growing interest in the study of this type of equations, particularly in understanding the similarities and differences between the behavior of its solutions and those of the local version. Particulary, \cite{Berestycki2009,Hamel2014,Penington2018} consider the equation
	\begin{equation} \label{eq: nonlocal}
		u_t-\Delta u =\mu u(1-\phi*u), 
	\end{equation}
	where $\phi\in L^1(\mathbb{R})$ is a given convolution kernel satisfying  $\int_{\mathbb{R}}\phi(x)dx=1$.
	These authors studied the existence and stability of steady states and traveling wave solutions for certain choices of $\mu$ and $\phi$, however, these aspects are not the focus of the present work. From the particle‑system literature, it has been shown that equation~\eqref{eq: nonlocal} arises as the hydrodynamic limit of a branching Brownian motion with decay of mass \cite{AddarioBerry2017,AddarioBerry2019}, that is, a standard branching Brownian motion with added local competition. This connection was used to derive results on the position of the front of the particle system.
\end{remark}

\begin{remark}
	The proof of Theorem \ref{thm: convemp} is based on techniques and ideas adapted from the works of  \cite{Fournier2004},  \cite{Champagnat2007},  \cite{Meleard2015}, and \cite{Beckman2019}, among others. 
	The dynamics of the Log-BBM can be viewed as a particular case of the nonlocal Lotka–Volterra cross-diffusion systems introduced by Mel\'eard and Fontbona in \cite{Fontbona2013}. However, while Mel\'eard and Fontbona  consider general Lipschitz continuous interaction kernels, our interaction kernel is given by $K(x) = c\mathbb{1}_{\{x\leq 0\}}$, which is not Lipschitz. Fontbona and Muñoz-Hernández \cite{Fontbona2022} studied a closely related model in which particles move according to a more general diffusion process, reproduction is homogeneous in space, and logistic competition is uniform throughout space. In this process, the population size evolves as in the logistic branching Brownian motion, but there is no selective advantage for particles with higher positions. They established a convergence‑rate result for the particle system toward the large‑population limit proved in \cite{Fontbona2013}, which in this particular case, the reaction term is also non-local.
\end{remark}

Let  
\begin{equation}\label{eq: cdf0}
	F^K(t,x):=\mu^K_t((-\infty,x]), \qquad t\geq 0, x\in\mathbb{R}
\end{equation}
be the (random) cumulative distribution function of the random measure given by \eqref{eq: defemp0}. As a consequence of Theorem~\ref{thm: convemp}, we have that for any $t \geq 0$,  $F^K(t,\cdot)$ converges in distribution to $F(t,\cdot)$, the cumulative function of  the solution of \eqref{eq: eqmu0}. 
As $\partial_x F(t,x) = u(t,x)$, for every test function $\phi\in C^{1,3}_c(\mathbb{R}^+\times\mathbb{R})$, the function $F$ satisfies the weak form of the FKPP equation, namely
\begin{equation}\label{eq: weakFKPP}
	\begin{split}
		\int_{\mathbb{R}}F\partial_x \phi   dx = \int_{\mathbb{R}}F_0\partial_x \phi_0   dx+ \int_0^t\int_{\mathbb{R}}\Big(\partial_{t}\partial_{x} \phi + \partial_{xxx} \phi  - (1- F )\partial_x \phi \Big)Fdxds.
	\end{split}
\end{equation}
Note that, applying the same tools as in Lemma \ref{lem: uniqueness}, we can show that the equation above admits a unique solution on $[0,T]$. In \cite{Kolmogorov1937} it was shown that a classical solution exists for the FKPP equation \eqref{eq: fkpp} with Heaviside initial condition, that is a function $u\in C^{1,2}([0,T]\times \mathbb{R})$ that satisfies equation \eqref{eq: fkpp} with forcing term $H(u)= u(1-u)$. As a classical solution is a solution of the same equation in its weak form, we have that it must coincide with $F$ on the interval $[0,T]$.
Therefore, we can deduce the following result.   
\noindent
\begin{corollary}\label{Col: cdf}
	For any $T>0$, the cumulative distribution function associated with the empirical measures defined in \eqref{eq: defemp0} converges in probability to a deterministic limit $F(t,\cdot)$, which is a weak solution of the FKPP equation, i.e. for a  given initial CDF $F_0$,
	\begin{equation}
		\left\{
		\begin{aligned}
			\partial_t F (t,x)  &=  \frac{1}{2}\partial_{xx} F (t,x)  - F(t,x)  (1-  F(t,x)), \qquad t\leq T, x\in\mathbb{R}\\
			F(0,x) & =  F_0(x), \qquad x\in\mathbb{R}         
		\end{aligned}
		\right..
	\end{equation}
	Furthermore, if $F_0(x) = \mathbb{1}_{\{x>0\}}$ and $U$ is a classical solution to the FKPP equation in the interval $[0,T]$, then $F$ coincides with $U$ in that time interval, i.e. $U(t,\cdot) = F(t,\cdot)$ for $ t\in[0,T]$.
\end{corollary}

\subsection{Microscopic weak selection principle}\label{weakselection}
The study of the FKPP equation has focused on the fact that it is a partial differential equation that admits traveling wave solutions. It is known that it has an unstable constant solution ($u\equiv0$) and a stable constant solution ($u\equiv 1$) (see \cite{Aronson1975}). 
Also, for initial conditions such that $u(0,x)$ goes to one as $x$ goes to $-\infty$ and $u(0,x)$ goes to zero as $x$ goes to $\infty$, there exists a one parameter family $\{U_v\}_{v}$ of traveling wave solutions, of the form $u(t,x) = U_v(x-vt),$ such that if we set $z= x-vt$ then $U_v(z)\to1$ as $z\to-\infty$ and $U_v(z)\to 0$ as $z\to\infty$. Moreover, it is known that when the initial condition ranges between one and zero, there exists traveling wave solutions only for velocities $v \geq \sqrt{2}$. In particular, for initial conditions that decays sufficiently fast, as 
\begin{align*}
	\int u(0,x)xe^{\sqrt{2}x}dx< \infty,    
\end{align*}
the associated front travels at minimal speed $v_0= \sqrt{2}$ (see \cite{Bramson1983}). 

A classical special case exhibiting this behavior was already established in \cite{Kolmogorov1937}:
when the FKPP equation is started from a step (Heaviside-type) initial condition, the solution $u(t,x)$ develops into a traveling front. That is, there exists a non  trivial final shape $w_{\sqrt{2}}$ and a centering term $m(t)$ such that 
\begin{equation}\label{conv}
	u(t,m(t)+z) \underset{t\to\infty}{\longrightarrow} w_{\sqrt{2}}(z).
\end{equation}
In this case, the centering term is not unique, but it is usually taken such that $u(t,m(t))=1/2$. \cite{Bramson1978, Bramson1983} gave an entirely probabilistic proof that \begin{equation}\label{centering}
	m(t) = 	\sqrt{2}t - \frac{3}{2\sqrt{2}}\log t + O(1) \text{ when } t\to\infty, \text{ and }\quad w_{\sqrt{2}}(z) = Aze^{-\sqrt{2}z} \quad \text{for large} \ z.
\end{equation} 
On the other hand, if the initial condition behaves as $p(x)e^{-\sqrt{2}x}$ for some polynomial $p(x)$, the front still propagates at the minimal speed $v_0 = \sqrt{2}$, although the logarithmic correction in \eqref{centering} differs from the fast-decaying case. In turn, if the initial condition decays as $e^{-\gamma}$ for $\gamma< \sqrt{2}$ the front propagates with a speed larger than $\sqrt{2}$. 
Therefore, we can conclude that in all these cases there exists a limiting front profile $w_v(x)$ and a centering function $m(t)$ such that 
\[u(t,m(t)+z)\to w_v(z) \quad \text{and}\quad  \frac{m(t)}{t}\to v,\]
where, for large $z$  
\begin{equation}
	w_v(z)=\left\{\begin{array}{@{}l@{}}
		A_ve^{-\gamma(v)z} \ \text{if}\ v> \sqrt{2}  \\
		A_v z e^{-\sqrt{2}z}  \ \text{if}\ v= \sqrt{2}
	\end{array}\right.\,.
\end{equation}
Here,  $\gamma<\sqrt{2}$ is the smallest solution to the equation $\frac{\gamma}{2}+\frac{1}{\gamma}=v$ (see \cite{Brunet2016} for more on this subject).


Thus, we have seen that, when the initial condition decays sufficiently fast, deterministic PDEs select the minimal traveling wave speed and profile. 
As mentioned earlier, these noiseless equations typically represent a mean-field description of an underlying microscopic reaction–diffusion system (as seen in subsection \ref{nexus}). The study of how the introduction of a microscopic amount of random noise affects this mean-field picture was pioneered by  Brunet and Derrida. 
They considered two main approaches: adding a multiplicative noise term to the equation to represent external noise, rather than fluctuations inherent to the microscopic model, and working directly with the microscopic model itself. 
Independent of the approach,  Brunet and Derrida \cite{Brunet1997,Brunet1999,Brunet2001} found that, in contrast to the macroscopic equations, which admit a continuous family of traveling wave solutions, the microscopic models select a single front velocity (the minimal) for arbitrary initial conditions as one approaches the mean field limit. In the literature, this phenomenon is known as the \emph{microscopic weak selection principle}. 

The first weak microscopic selection principle was proven by Bramson et. al in \cite{Bramson1986}. 
They studied a model on the lattice with state space $\Omega=\{0,1\}^{\mathbb{Z}}$, where at any time a point can be empty or occupied by at most one particle. 
Particles can jump or create new ones only at the nearest neighbor (if it is empty), at a rate parameterized by some parameter $\theta$ . 
The authors show that this model, as $\theta$ goes to infinity, has a hydrodynamic limit given by a reaction-diffusion equation. 
They also showed that, for all $\theta< \infty$, the system has a unique invariant distribution when seen from its rightmost particle. Furthermore, the average velocity of this particle  converges to the minimal wave speed of the corresponding reaction-diffusion equation. 
Following ideas from \cite{Bramson1986}, Groisman and Soprano- Loto \cite{Groisman2021} proved a microscopic weak selection principle for their generalization of the $N$-BBM.

Along the same lines, we were able to establish the weak selection principle for our model and to show that the asymptotic velocity of the system coincides with the speed of the rightmost particle in the BBM. Our approach builds on the ideas in  \cite{Groisman2020} and \cite{Groisman2021}, making use of the monotonicity of the spacings between the particles ordered by fitness  within the system. 

Let $\xi_{N_0^c}$ denote some initial configuration of particles, where $N_0^c \sim \pi$. The following result then follows, the complete proof of which can be found in Section~\ref{microweak}.
\begin{theorem} \label{thm: speed}
	Given the Log-BBM($1,c$) $(\boldsymbol{X^c}(t))_{t\geq 0}$,
	\begin{itemize}
		\item[i)] \label{it: vmaxequalvmin} there exists $v_{c} \in \mathbb{R}$ such that for every $\xi_{N_0^c}$,   
		\begin{equation*}
			\lim_{t\rightarrow\infty} t^{-1}\max_{u\in\mathcal{N}^c_t} X^c_u(t) = \lim_{t\rightarrow\infty} t^{-1}\min_{u\in\mathcal{N}^c_t} X^c_u(t) =v_{c} \quad \mathbb{P}_{\xi_{N_0^c}}\text{-a.s and in } \ L^1.
		\end{equation*}
		\item[ii)]  \label{it: speed0}       Furthermore, we have that 
		\begin{equation*}
			v_0:=\lim_{c\rightarrow0} v_{c} = \sqrt{2}.
		\end{equation*}   
	\end{itemize}    
\end{theorem}
\noindent 
The paper is organized as follows. In Section~\ref{Log}, we introduce the model and provide an alternative construction via a coupling with branching Brownian motion. In Section~\ref{hydrodynamic}, we prove Theorem~\ref{thm: convemp}, namely the convergence of the empirical measure of the log-BBM to a probability measure whose density solves the nonlocal PDE \eqref{eq: densitymu0}. Finally, in Section~\ref{microweak}, we prove Theorem~\ref{thm: speed}. In particular, Subsection~\ref{subsec: existence} establishes the existence of the limiting velocity, while Subsection~\ref{subsec: speedselection} shows that, as $c \to 0$, this limiting speed converges to the minimal velocity of the FKPP travelling wave solution.

\section{Logistic Branching Brownian Motion}\label{Log}

\subsection{Branching process with logistic growth}\label{LB}
Let $(N_t)_{t\geq 0}$ be a \emph{branching process with logistic growth (LB-process)}, that is a continuous time Markov chain taking values in $\mathbb{N}_0$, with the following dynamics.
\begin{itemize}
	\item At rate one, each particle produces $k$ children with probability $\rho_{k}$.
	\item At rate $d$, each particle dies naturally.
	\item At rate $c$, each particle selects another fixed particle  and kills it.	 
\end{itemize}
Therefore, when the total population has size $n$, the transitions of the process are given by
\begin{equation*}
	n \mapsto \left\{
	\begin{array}{cc}
		n+k     &\text{at rate } n\rho_k,  \qquad  \text{if } k \in \mathbb{N}\\
		n-1     &\text{at rate } n(d + c(n-1))  
	\end{array}	
	\right. . 
\end{equation*}
We will assume that the logistic branching process does not explode in finite time a.s., i.e.
\begin{equation} \label{cond L} \tag{Condition L}
	\sum\limits_{k\in \mathbb{N}} \rho_k \log (k) < \infty. 
\end{equation}
Under this condition, Lambert \cite{Lambert2005} showed that when $d>0$ the process goes extinct almost surely. On the other hand, when $d = 0$ (i.e., in the absence of natural deaths), the LB-process is positive recurrent in $\mathbb{N}$ and converges in distribution to a stationary measure $\pi$. From now on we will work with the dyadic case, that is $\rho_1 = 1$ and $d=0$. In this case, the stationary distribution $\pi$ takes the form 
\begin{equation}\label{eq: mu}
	\pi_k=\dfrac{c^{-k}}{ k!}\dfrac{e^{-1/c}}{1 - e^{-1/c}} \qquad k\in\mathbb{N},
\end{equation} 
which correspond to a Poisson random variable with mean $c$ conditioned to be positive, that we denote by $N^c_0$. Note that 
\begin{equation}\label{eq: exp}
	\mathbb{E}[N^c_0]= \dfrac{1}{c(1-e^{-1/c})}.
\end{equation}
\begin{remark}\label{re: nacpuro}
	Note that in the absence of competition, that is $c=0$, the LB-process defined above is a continuous-time Bienayme-Galton-Watson (BGW) process. If in addition $d=0$ and $\rho_1=1$, it becomes a Yule process.    
\end{remark}

\subsection{Definition of the Logistic Branching Brownian Motion with selection}
Our object of study will be a modification of the $N$-BBM model with the following dynamics:
\begin{itemize}
	\item At time $t=0$ we have an amount of particles given by the random variable $N^c_0$ with distribution $\pi$, given by \eqref{eq: mu}. The particles are positioned on the real line according to a (spatial) density $\boldsymbol{\psi}_{\pi}$.
	\item Each particle moves independently as a standard Brownian motion from this initial configuration.
	\item (\emph{Reproduction event.}) At rate one, each particle has a new offspring, that moves independently as standard Brownian motion from its birthplace.
	\item (\emph{Competition event.}) At rate $c$, each particle selects another particle, and the one with the \emph{lower fitness}, determined by its position, is removed from the system. Formally, let $\mathcal{N}^c(t)$ denote the set of indexes of alive particles at time $t$ and let $X^c_u(t)$ be the position of the particle indexed by $u$. Therefore, if the system selects particles  $u, v \in \mathcal{N}^c(t)$ to compete, the one that gets removed is the particle $k_{u,v}(t) := \argmin_{u,v}\{X^c_u(t),X^c_v(t)\}.$ 
\end{itemize}
Clearly, the process $(|\mathcal{N}^c(t)|)_{t\geq 0}$ correspond to a LB-process, and we refer to the process $(\mathbf{X}^c(t))_{t\geq 0}$, given by $\mathbf{X}^c(t) = \Big(X^c_u(t), u\in\mathcal{N}^c_t\Big)$, described above as a (dyadic) \textbf{logistic branching Brownian motion with selection} with branching rate one and competition rate $c$. In the remainder of the paper, we will refer to it as \textbf{Log-BBM$(1,c)$}.  
\noindent
Indeed, the Log-BBM$(1,c)$ is a Markov process in $\mathbb{S}=\bigcup\limits_{n\in \mathbb{N}} \mathbb{R}^n $, with infinitesimal generator given by   
\begin{equation}\label{InfGenLogBBM}
	\mathcal{G} F(\boldsymbol{x}) = \frac{1}{2} \sum_{i=1}^{n} \frac{\partial^2}{\partial x_i^2} F(\boldsymbol{x}) +  \sum_{i=1}^{n}\big[ F(\Pi^{+}_{i}(\boldsymbol{x})) - F(\boldsymbol{x})\big] +c \sum_{i=1}^{n}\sum_{i\neq j} \big[F(\Pi^{-}_{ij}(\boldsymbol{x})) - F(\boldsymbol{x})\big]
\end{equation}
for every $F:  \mathbb{S} \rightarrow \mathbb{R}$ with $F \in C^2(\mathbb{S})$. Here, for each $\boldsymbol{x} \in \mathbb{R}^n \subset \mathbb{S}$ and $i \in \{1,\dots,n\}$, $\Pi^{+}_i(\boldsymbol{x})$ is the configuration obtained from $\boldsymbol{x}$ by replacing $x_i$ with two particles at the same location, that is
\begin{equation}
	\Pi^{+}_i ((x_1,\cdots,x_n)) := (x_1,\cdots,x_i,x_i,\cdots,x_n) \in \mathbb{R}^{n+1};
\end{equation}
and for $i,j \in \{1,\dots,n\}$, $\Pi^{-}_{ij} (\boldsymbol{x})$ correspond to the configuration obtained from $\boldsymbol{x}$ by removing $k_{i,j}(t)$, that is
\begin{equation}
	\Pi^{-}_{ij} ((x_1,\cdots,x_n)) := (x_1,\cdots,x_{k_{i,j}-1},x_{k_{i,j}+1},\cdots,X_n) \in \mathbb{R}^{n-1}.
\end{equation}
\subsubsection{Construction of the Log-BBM using the BBM}\label{subsec: coupling}

Let $\mathcal{T}$ be a continuous-time dyadic tree, associated with a Yule process. We label the nodes of the tree using the Ulam-Harris notation, in the space   
\begin{align*}
	\mathcal{U}:= \bigcup_{n\geq0}\mathbb{N}^n ,
\end{align*}
where $\mathbb{N}^0= \{\emptyset\} $. For any element $u  = (u_1, \cdots, u_n) \in \mathcal{U}$ we write $|u|= n$ for the generation of $u$, and for $0\leq k< n$, $u|k = (u_1,\cdots, u_k)$ is the $k$-th ancestor of $u$. Particularly, $u|0=u_0$ is the root.  For any $u\in \mathcal{T}$, we denote by $\tau^{b}_u$ the birth time of particle $u$ ($\tau_{u_0}^b = 0$) and we can assign an independent standard Brownian motion $B_u = \{B_u(s): s\geq 0\}$, which describes the motion of $u$ relative to its birth-position. With this notation we can write the position of $u$ at time $t\geq \tau_u^b$ by 
\begin{equation} \label{eq: position}
	X_u(t) = X_{u||u|-1}(\tau^b_{u}-) + B_u(t-\tau^b_{u})= X_{u_0}(0) + \sum_{k=0}^{|u|-1} B_{u|k}(\tau^b_{u|k+1}-\tau^b_{u|k}) + B_u(t-\tau^b_{u}).     
\end{equation}
On the other hand, we will abuse the notation used to denote the position of $u$ at time $t< \tau_u^b$ by $X_u(s) :=X_v(s)$, where $v$ is the (only) ancestor of $u$ at time $s$. 
If for each $u$, $\tau^b_{u|k+1}-\tau^b_{u|k}\sim\text{Exp}(1)$ for all $k< |u|$, the Markov process $\big(\mathbf{X}(t)\big)_{t\geq 0}: = \big(X_u(t), u\in\mathcal{N}_t\big)_{t\geq 0}$ is the dyadic branching Brownian motion, where $\mathcal{N}_t:= \{u\in\mathcal{T}: \tau^b_u\leq t\}\subset \mathcal{U}$ is the set of particles alive at time $t$. In particular, $N_t :=|\mathcal{N}_t|$ follows a geometric distribution with parameter $e^{-t}$, so $\mathbb{E}(N_t) = e^t$. 

Consider now, an initial configuration $\mathbf{\xi}_{N_0^c}\in \mathbb{S}$ with $N_0^c\sim \pi$ and a system of $N_0^c$ independent dyadic branching Brownian motions $\big(\mathbf{X}^1(t)\big)_{t\geq 0}, \cdots,$ $\big(\mathbf{X}^{N_0^c}(t)\big)_{t\geq 0}$, starting at $\xi_{N_0^c}^i$ for $i\in[1,N_0^c]$ respectively. Then, we can define the set of particles alive at time $t$ on the system by 
\[\mathcal{M}_t := \bigcup_{n=1}^{N_0^c} \mathcal{N}^n_t.\]
Note that the process $(|\mathcal{M}_t|)_{ t\geq 0}$ is the  sum of $N_0^c$ i.i.d. geometric random variables with success parameter $e^{- t}$, that is, $|\mathcal{M}_t|\sim NB(N_0^c,e^{-t})$ conditional on $N_0^c$.

We can construct a coupling between the  Log-BBM($1,c$)  a system of $N^c_0$ dyadic BBMs. 
Indeed, consider the following dynamics: at time $t=0$ there are $N^c_0$ blue particles which start moving as Brownian motions. Each  particle give birth at rate one to a particle of the same color, that move from its birthplace as its parent. Additionally, at rate $c$ each blue particle selects another blue particle, and the leftmost particle changes its color to red.  Both particles evolve independently as BBMs with a branching rate one. This coupled construction gave us a Log-BBM($1,c$) (formed by the blue particles) and a BBM (formed by the red and blue particles), both starting with $N^c_0$ individuals (see Figure \ref{fi:fig2} below).Under this coupling, we have that
\begin{equation}
	|\mathcal{N}^c_t|\leq |\mathcal{M}_t|,\quad\text{ and } \quad\max_{u\in \mathcal{N}^c_t}X^c_u(t)\leq \max_{u\in \mathcal{M}_t}X_u(t)\quad \text{a.s. for}\quad t\geq 0. 
\end{equation}
Particularly, if we consider the logistic forest $\mathbb{T}^{c}:=\bigcup_{n=1}^{N_0^c}\mathcal{T}^{c,n}$, associated with the Log-BBM($1, c$), is easy to see that it is a sub-forest of $\bigcup_{n=1}^{N_0^c}\mathcal{T}^n$,  the forest associated with the $N_0^c$ dyadic BBMs. In addition to the birth time $\tau^{b}_u$, each particle $ u \in \mathbb{T}^{c}$  is assigned a death time  $\tau^{d}_u$. With this notation, 
\[
\mathcal{N}^c_t := \big\{u\in\mathbb{T}^{c}: \tau^b_u\leq t<\tau^d_u\big\} \subset \mathcal{N}_t,
\]
is the set of particles alive at time $t$, hence for  each $u\in\mathcal{N}^c_t $, their position at time $t$ is given by $X^c_u(t):=X_u(t)$, as defined in \eqref{eq: position}.

\begin{figure}[htbp]
	\centering 
	\includegraphics[scale=0.24]{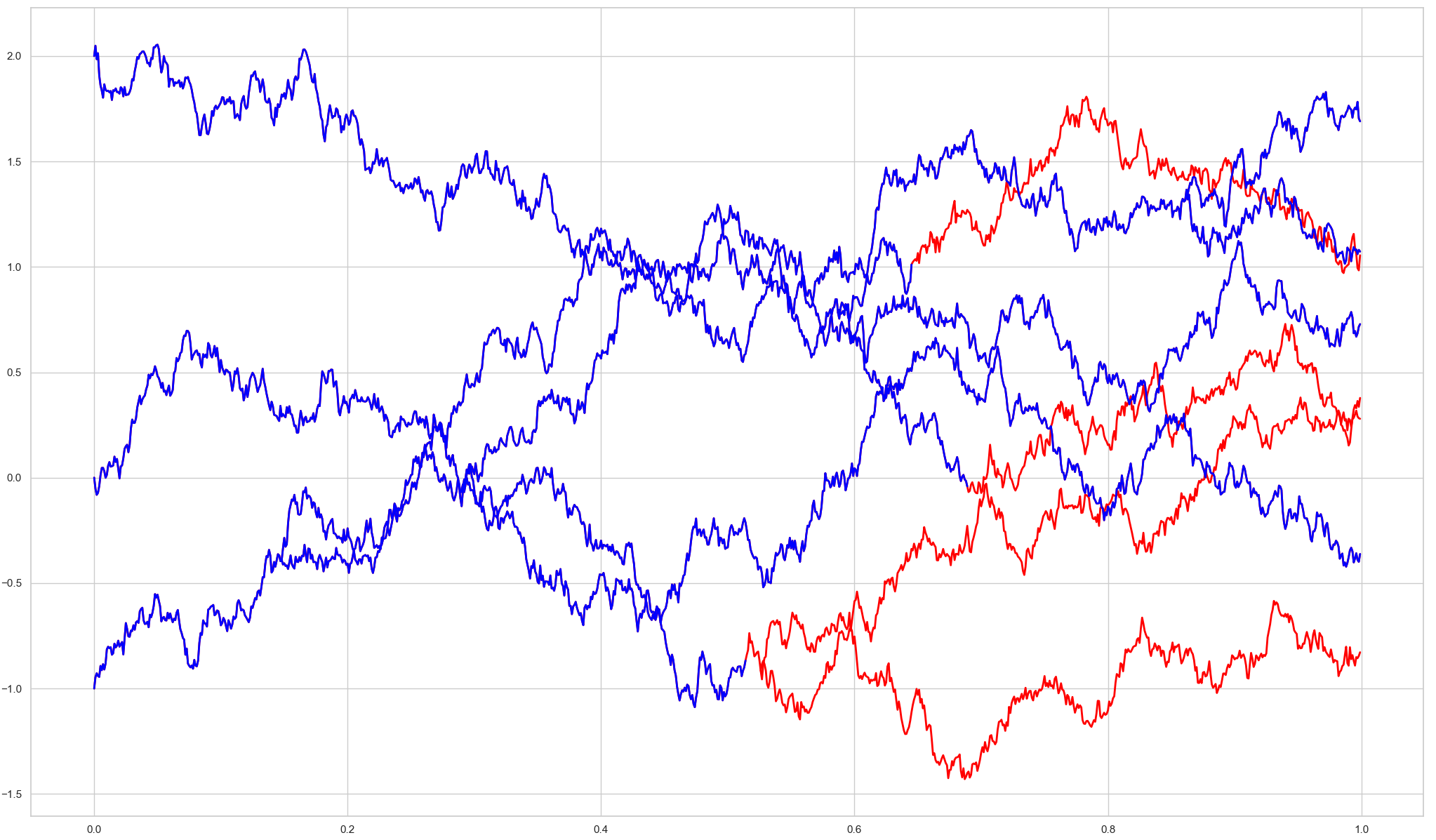}
	\caption{Simulation of three simultaneous BBMs with rate one and of the coupled Log-BBM$(1,0.25)$ on the interval $[0,1]$. Particles in blue belong to the Log-BBM (with selection) and all particles red and blue belong only to the original BBMs.}
	\label{fi:fig2}
\end{figure}

\section{Proof of Theorem \ref{thm: convemp}: Log-BBM as microscopic model for a Fisher-KPP type equation}\label{hydrodynamic}

We will start this section by describing the key elements of the proof of Theorem \ref{thm: convemp}.  First, in Lemma \ref{lem: martngl},  we show that  any measure $(\mu^{K}_t)_{t\in [0,T]}$ satisfies the SDE \eqref{eq: stochdiff}. Next, in Lemma \ref{lem: tighness}, we prove the tightness of the sequence of empirical measures $\big\{\left(\mu^{K}_t\right)_{t \in [0,T]}\big\}_{K\in\mathbb{N}} $, on the space $D\big([0,T], M^v_F(\mathbb{R})\big)$ defined in Notation \ref{not}, for any $T>0$. This ensures that any subsequence converges in law to some limit $(\mu_t)_{t\in[0,T]}$ in the same space. 
We then show that this limit measure, $(\mu_t)_{t\in[0,T]}$, solves equation \eqref{eq: eqmu0}, and is indeed its unique solution. Finally, we extend this result to the space $D\big([0,T], M_F(\mathbb{R})\big)$ and establish that  $(\mu_t)_{t\in [0,T]}$ is absolutely continuous and that its density satisfies \eqref{eq: densitymu0}. In what follows, the symbol $K$ is used as a superscript to represent a competition rate $c_K$.

Before stating the next lemma, let us first introduce some notation, using the construction from Subsection \ref{subsec: coupling}. We define the spatio‑temporal birth and death point measure of the process defined in $\mathbb{R}^+\times \mathbb{R}$ by
\begin{equation}\label{hl.12.1}
	\nu^b := \sum_{ u\in\mathbb{T}^K} \delta_{(\tau_u^b, X^K_u(\tau_u^b))},\qquad
	\nu^d := \sum_{ u\in\mathbb{T}^K }\delta_{(\tau_u^d, X^K_u(\tau_u^d))},
\end{equation}
respectively. For any $A\in\mathcal{B}(\mathbb{R}^+\times \mathbb{R})$, $\nu^b(A)(\text{resp.} \ \nu^d(A))$ counts the number of birth (resp. death) events occurring in $A$, i.e. for $A = [0,t]\times B$, $\nu^b(A) =\# \big\{u\in\mathbb{T}^K: \tau_u^b\leq t, X^K_u(\tau_u^b)\in B\big\}$ (resp. $\nu^d(A) =\# \big\{u\in\mathbb{T}^K: \tau_u^d\leq t, X^K_u(\tau_u^d)\in B\big\}$). With the notation introduced above, we can now state our first lemma, which provides a semimartingale representation of 
$\big<\phi,\mu^K\big>$ for every test function $\phi \in C_b^{1,2}(\mathbb{R}^+\times \mathbb{R})$ 
\begin{lemma}\label{lem: martngl}
	For any $K \in \mathbb{N}$, the measure $(\mu^{K}_t)_{t\geq 0}$ defined in \eqref{eq: defemp0} satisfies the following stochastic differential equation
	\begin{equation}\label{eq: stochdiff}
		\begin{split}
			\big< \phi, \mu^{K}_t\big>  =&  \big< \phi, \mu^{K}_0\big>
			+\int_{0}^{t}\int_{\mathbb{R}}  \Big(\partial_t\phi(s, x) +\frac{1}{2} \partial_{xx}\phi(s, x ) \Big) \mu^{K}_s(dx)ds\\
			&+\int_0^t \int_{\mathbb{R}}\phi(s,x)\Big(1-2\int_x^\infty\mu^{K}_s(dy)+2c_K\Big)\mu^{K}_s(dx)ds +L^{K,\phi}_t 
		\end{split},
	\end{equation}
	for all bounded functions $\phi(t,x) \in C^{1,2}_b(\mathbb{R}_+\times\mathbb{R})$. Here, $(L^{K,\phi}_t )_{t\geq 0}$ is a c\`adl\`ag martingale defined by 
	\begin{align*}
		L^{K,\phi}_t : = B^{K,\phi}_t + M^{K,\phi}_t  - D^{K,\phi}_t,
	\end{align*} 
	where $(B^{K,\phi}_t )_{t\geq 0}$, $(D^{K,\phi}_t )_{t\geq 0}$ and $(M^{K,\phi}_t)_{t\geq 0}$ are square-integrable martingales given by 
	\begin{equation}\label{eq: martings}
		\begin{split}
			&B^{K,\phi}_t  = \int_0^t\int_{\mathbb{R}} \phi(s,x)\nu^b(ds,dx) - \int_0^t \frac{1}{m_K}\sum_{u \in \mathcal{N}^K_s}\phi(s,X^K_u(s))ds \, ;\\
			& D^{K,\phi}_t  = \int_0^t\int_{\mathbb{R}} \phi(s,x)\nu^d(ds,dx) - \int_0^t  \frac{2c_K}{m_K}  \sum_{u \in \mathcal{N}^K_w} \phi(w,X^K_u(w))\Bigg(\sum_{v \in \mathcal{N}^K_w}\mathbb{1}_{\{X^K_u(w)\leq X^K_v(w)\}}-1\Bigg)dw \, \text{; and}\\
			& M^{K,\phi}_t  = \frac{1}{m_K}\sum_{u\in\mathbb{T}^K: \tau^b_u \leq t }\int_{0}^{t} \mathbb{1}_{\{\tau^{b}_u\leq s< \tau^{d}_u\}}\partial_x\phi(s, X^K_u(s)) dX^K_u(s), 
		\end{split}
	\end{equation}
	respectively. Furthermore, $L^{K,\phi} $ has  predictable quadratic variation 
	\[
	\begin{split}
		\big<L^{K,\phi} \big>_t & = \int_0^t  \frac{2c_K}{m^2_K}  \sum_{u \in \mathcal{N}^K_w} \phi^2(w,X^K_u(w))\Bigg(\sum_{v \in \mathcal{N}^K_w}\mathbb{1}_{\{X^K_u(w)\leq X^K_v(w)\}}-1\Bigg)dw \\&- \int_0^t \frac{1}{m_K^2}\sum_{u \in \mathcal{N}^K_s}\phi^2(s,X^K_u(s))ds - \frac{1}{m_K^2}\sum_{u\in\mathbb{T}^K: \tau^b_u \leq t}\int_{0}^{t} \mathbb{1}_{\{\tau^{b}_u\leq s <\tau^{d}_u\}}(\partial_x\phi(s, X^K_u(s)))^2 ds
	\end{split},\]
	and converges to $0$ when $K$ goes to infinity in $L^2$.
\end{lemma}
\noindent
The proof is omitted here and can be found in Appendix \ref{sec: AppdxA}. The following lemma will help us extend convergence under the vague topology to convergence under the weak topology. The proof is postponed to Appendix \ref{appendix: vagueapprox}

\begin{lemma}\label{lem: vagueapprox}
	Consider the function $\psi(x) := 6x^5-15x^4+10x^3$ and define a sequence of functions 
	$\{\psi_n\}_{n\geq 0}\subset C^2_b(\mathbb{R})$ such that 
	\[\psi_n(x): = \psi(0\vee(|x|-(n-1))\wedge1),\qquad n\in \mathbb{N}.\]
	For every $T>0$, the following holds
	\begin{itemize}
		\item[a)]\label{it: a} For every $p\in\mathbb{N}$, there exists a constant $\tilde{C}:=\tilde{C}(p,T)$ such that
		\[\sup_K\mathbb{E}(\sup_{t\leq T}K^{-p}(N_t^K)^p)\leq \tilde{C}\sup_K [K^{-p}\mathbb{E}(N_0^K)^p].\]
		\item[b)]\label{it: b} Let $\mu_t^K$ be a convergent subsequence of the sequence defined by \eqref{eq: defemp0}, then we have   
		\begin{equation}
			\lim_{n\to \infty}\sup_{K} \mathbb{E}\Big(\sup_{t\in[0,T]}\big<\mu^K_t,\psi_n\big>\Big)=0.
		\end{equation}
		\item[c)]\label{it: c} Furthermore, if $u_t$ is the limit in distribution of $\mu_t^K$ in the space $D([0,T],M_F^v(\mathbb{R}))$ then
		\begin{equation}
			\lim_{n\to \infty} \mathbb{E}\Big(\sup_{t\in[0,T]} \big<\mu_t,\psi_n\big>\Big)=0.
		\end{equation}
	\end{itemize}
	
\end{lemma}
\noindent
We now establish the tightness of the rescaled empirical measures in the following lemma
\begin{lemma}\label{lem: tighness}
	For any $T>0$, the sequence of rescaled empirical measures $\big\{\big(\mu^{K}_t\big)_{t\in[0,T]}\big\}_{K\in\mathbb{N}} $, defined by equation \eqref{eq: defemp0}, is tight on the space  $D\big([0,T], M^v_F(\mathbb{R})\big)$. 
\end{lemma}

\begin{proof}
	For ease of notation, we will denote $\mu^{K}:= \big(\mu^K_t\big)_{t\in[0,T]}$. By \cite[Theorem 2.1]{Roelly‐Coppoletta01041986}, the sequence $\big\{\mu^{K}\big\}_{K\in\mathbb{N}}$ is tight in $D\big([0,T], M^v_F(\mathbb{R})\big)$ if the sequence $\{\big<g_n, \mu^K\big>\}_{K \in \mathbb{N}}$ is tight in $D\big([0,T], \mathbb{R}\big)$ for each $n\in \mathbb{N}$, where $\{g_n\}_{n \in \mathbb{N}}$ is a dense sequence of functions in $C_0( \mathbb{R})$. 
	By Lemma \ref{lem: martngl}, we know that $\big<f, \mu^K\big>$ has a semimartingale representation, for every $f\in C^{1,2}_b(\mathbb{R}_+\times\mathbb{R})$ and $K\in\mathbb{N}$. In particular, for any  $f\in C^{2}_b(\mathbb{R})$ we have that 
	\begin{align*}
		\big<f, \mu^K_t\big> = \big<f, \mu^K_0\big> + A^{K,f}_t + L^{K,f}_t,
	\end{align*}
	where $\big(L^{K,f}_t\big)_{t\in [0,T]}$ is as defined in Lemma \ref{lem: martngl} and the finite variation process  $A_t^{K,f}$ is given by
	\begin{equation}\label{eq : At}
		A^{K,f}_t  :=  \int_{0}^{t}\int_{\mathbb{R}}  \Bigg[\frac{1}{2} \partial_{xx}f(x) + f(x)\Big(1 - 2\int_x^\infty\mu^{K}_s(dy) + 2c_{K}  \Big)\Bigg]\mu^{K}_s(dx)ds. 
	\end{equation}
	\noindent
	Then, by Rebolledo’s criterion \cite[Corollary 2.3.3]{Joffe1986}, to ensure the tightness of the sequence $\big\{\big<f, \mu^K\big>\big\}_{K \in \mathbb{N}}$ in $D\big([0,T], \mathbb{R}\big)$ it is sufficient to show that the sequences  
	$\{A^{K,f}\}_{K\in\mathbb{N}}$ and $\left\{\left<L^{K,f}\right>\right\}_{K\in\mathbb{N}}$ are tight in $D\big([0,T],\mathbb{R}\big)$. Moreover, by Aldous’s tightness criterion (see \cite[Subsection $2.2$]{Joffe1986}), this is guaranteed provided the following conditions hold:
	
	\begin{itemize}
		\item[1.]
		For a dense subset of times, the family $\big\{\big<f, \mu^K_t\big>\big\}_{K\in \mathbb{N}}$ is tight in $\mathbb{R}$; that is, for each $t \in [0,T]\cap\mathbb{Q}$ and every $\varepsilon>0$, there exist a real number $M=M(\varepsilon)>0$ such that
		\begin{align*}
			\sup_{K \in \mathbb{N}} \mathbb{P}\Big(\big| \big<f, \mu^K_t\big> \big|>M\Big)\leq \varepsilon.
		\end{align*}
		
		\item[2.]
		For any collection of stopping times $\{\tau_K\}_{K\in \mathbb{N}}$ that are almost surely bounded by $T$, and every $\varepsilon>0$ 
		\begin{equation}\label{eq: AldousA}
			\lim\limits_{\delta \to 0} \limsup\limits_{K \to \infty}\sup\limits_{\tau_K \leq T} \sup\limits_{h\in [0,\delta]}\mathbb{P}\Big(\big|A_{\tau_K + h}^{K,f} - A_{\tau_K}^{K,f}\big|>\varepsilon\Big)=0,
		\end{equation}
		and 
		\begin{equation}\label{eq: AldousL}
			\lim\limits_{\delta \to 0} \limsup\limits_{K \to \infty}\sup\limits_{\tau_K \leq T} \sup\limits_{h\in [0,\delta]}\mathbb{P}\Big(\big|\big<L^{K,f}_{\tau_K+h}\big> - \big<L^{K,f}_{\tau_K}\big>\big|>\epsilon\Big)=0.
		\end{equation}
	\end{itemize}
	\noindent
	The first condition follows by taking $M= \|f\|_{\infty}/\varepsilon$. We can focus on conditions \eqref{eq: AldousA} and \eqref{eq: AldousL}. Note that for each  $ K \in \mathbb{N}$, $K^{-1}N^K_{t}=\big<1, \mu^K_t\big>$, then there exists constants $C$ and $\tilde{C}$ such that 
	\[\begin{split}
		\mathbb{E}\Big(\Big|A_{\tau_K + h}^{K,f} - A_{\tau_K}^{K,f}\Big|\Big)
		\leq hc\sup_K\Bigg(C\mathbb{E}\Big[\sup_{t\leq T}K^{-1}N^K_{t}\Big]+ 2c\|f\|_\infty\mathbb{E}\Big[\sup_{t\leq T}K^{-2}(N^K_{t})^2\Big]\Bigg),
	\end{split}\]
	and   
	\[\begin{split}
		\mathbb{E}\Big(\Big|\big <L^{K,f}_{\tau_K+h}\big> - \big<L^{K,f}_{\tau_K}\big>\Big|\Big)\leq \frac{c^2h}{K}\sup_K\Bigg(\widetilde{C}\mathbb{E}\Big[\sup_{t\leq T}K^{-1}N^K_{t}\Big]+ 2c\|f^2\|_\infty\mathbb{E}\Big[\sup_{t\leq T}K^{-2}(N^K_{t})^2\Big]\Bigg).
	\end{split}\]  
	Using the Markov inequality and Lemma \ref{lem: vagueapprox}, we can conclude that both conditions \eqref{eq: AldousA} and \eqref{eq: AldousL} hold and consequently $\big\{\big<f, \mu^K\big>\big\}_K$ is tight in $D\left([0,T], \mathbb{R}\right)$ for each $f \in C_b^{2} (\mathbb{R})$. We also know that any function $g\in C_b(\mathbb{R})$ can be obtained as the bounded pointwise limit of a sequence of functions 
	$f_n\in C_b^2(\mathbb{R})$. Hence, we can conclude that the sequence $\big\{\big<g, \mu^K\big>\big\}_K$ is tight on $D\left([0,T], \mathbb{R}\right)$ for each $g \in C_b (\mathbb{R})$. Therefore, $\big\{\big<f, \mu^K\big>\big\}_K$ is tight for a dense sequence $\{g_n\}_{n \in \mathbb{N}}\subset C_0(\mathbb{R})$, and we obtain the desired result.
\end{proof}

For the next lemma, we consider a different representation of equation \eqref{eq: eqmu0}. For $t\leq T$ and $x\in\mathbb{R}$ we define the process $X_t := x+ B_t$, where $(B_t)_{t\geq 0}$ is a standard Brownian motion, and denote by $\mathbb{P}_x$ its law. The associated transition semi-group $(P_t)_{t\geq 0}$ is given by $P_t(x,dy) =\mathbb{P}_x (X_t\in dy)$. Now, for each $s\in [0,t]$ and $f\in C_b(\mathbb{R})$, let $\phi(s,x) = P_{t-s}f(x) = \mathbb{E}_x(f(X_t))$ . By the Feynman-Kac formula, the function $\phi(s,x)\in C^\infty_b([0,t]\times \mathbb{R})$ satisfies $\partial_t\phi(s,x) + \frac{1}{2}\partial_{xx}\phi(s,x)=0$ with final condition $\phi(t,x) = f(x)$. Substituting  this $\phi(s,x)$ into \eqref{eq: eqmu0} we obtain
\begin{equation}\label{eq: mild}
	\begin{split}
		\big<f, \mu_t\big> = \big<P_{t}f, \mu_0\big> + \int_0^t \int_{\mathbb{R}} P_{t-s}f(x)\Big(1  - 2  \int_x^{\infty} \mu_s(dy)\Big)\mu_s(dx)ds
	\end{split}  
\end{equation}

\begin{lemma}\label{lem: uniqueness}
	Let $\mu_0 \in M_F(\mathbb{R})$ be an initial condition.  Let 
	$\mu^1$ and $\mu^2$ be two solutions of \eqref{eq: eqmu0} in $C\left([0,T], M_F(\mathbb{R})\right)$ starting from $\mu_0$. Then, $\mu^1 =\mu^2$ on $[0,T]$.
\end{lemma}
\begin{proof}
	Assume $\mu^1, \mu^2 \in C\left([0,T], M_F(\mathbb{R})\right)$ be two solutions of \eqref{eq: eqmu0} such that $\sup_{t\leq T} \left<1,\mu_t^1+\mu_t^2\right>= C_T<\infty$. We define the variation norm by 
	\[\|\mu^1_t - \mu^2_t\|_{TV} := \sup_{\|f\|_{\infty}\leq 1}\big|\big<f, \mu^1_t - \mu^2_t\big>\big|,\]
	where $\|f\|_{\infty} = \esssup_{x\in \mathbb{R}}|f(x)|$. We want to show that $\|\mu^1_t - \mu^2_t\|_{TV} = 0$ for all $t \in [0,T]$. To prove this, let $g \in C_b(\mathbb{R})$ be such that $\|g\|_{\infty}\leq 1$, this implies that  $\|P_{t-s}g\|_{\infty}\leq 1$. For every $t\leq T$, by \eqref{eq: mild} we have that
	\[
	\begin{split}
		\Big|\big<g,\mu_t^1 -\mu_t^2 \big>\Big| =  \Bigg|&\int_0^t\int_{\mathbb{R}}  P_{t-s}g(x)\Big(1-2\int_x^{\infty} \mu^2_s(dy)\Big)\mu^2_s(dx)ds \\&- \int_0^t\int_{\mathbb{R}} P_{t-s}g(x) \Big(1-2\int_x^{\infty} \mu^1_s(dy)\Big)\mu^1_s(dx)ds\Bigg|
	\end{split}.
	\]
	Adding and subtracting $\int_0^t\int_{\mathbb{R}}  P_{t-s}g(x)\Big(1-2\int_x^{\infty} \mu^1_s(dy)\Big)\mu^2_s(dx)ds$ we obtain the following
	\[
	\begin{split}
		\Big|\big<g,\mu_t^1 -\mu_t^2 \big>\Big| \leq  &\int_0^t\Bigg|\int_{\mathbb{R}} 2P_{t-s}g(x)\int_x^{\infty} \big[\mu^1_s(dy) - \mu^2_s(dy)\big]\mu^2_s(dx)\Bigg|ds \\+ &\int_0^t\Bigg|\int_{\mathbb{R}} P_{t-s}g(x)\Big(1-2\int_x^{\infty} \mu^1_s(dy)\Big)\big[\mu^2_s(dx)-\mu^1_s(dx)\big]\Bigg|ds\\
		& \leq 2C_T\|P_{t-s}g\|_{\infty} \int_0^t \|\mu^2_s-\mu^1_s\| ds+ \|P_{t-s}g\|_{\infty}(1+C_T)\int_0^t \|\mu^2_s-\mu^1_s\|_{TV} ds
	\end{split},
	\]
	taking the supremum over all functions $g$ continuous and bounded, we have 
	\[
	\sup_{g\in C_b(\mathbb{R}),\|g\|_{\infty}\leq 1}|\big<g, \mu_t^1 -\mu_t^2\big|  \leq 
	\big(1+ 3C_T \big) \int_0^t\|\mu^1_s-\mu^2_s\|_{TV}ds.
	\]
	Consider now a general function such that $\|f\|_{\infty}\leq 1$. By a corollary of Lusin's theorem (see \cite[Corollary $2.24$]{Rudin1986}) we know that there exists a sequence $(g_n)_n \subset C_c(\mathbb{R})$ satisfying $\|g_n\|_{\infty}<1$ for every $n$, such that $f(x) = \lim_{n\to\infty}g_n(x)$ for almost every $x\in\mathbb{R}$. Therefore, by the dominated convergence theorem, we have
	\begin{equation}
		\sup_{f \in L^{\infty}(\mathbb{R}),\|f\|_{\infty}\leq 1}|\big<f, \mu_t^1 -\mu_t^2\big>  \leq 
		\big(1+ 3C_T \big) \int_0^t\|\mu^1_s-\mu^2_s\|_{TV}ds
	\end{equation}
	Finally, by Gronwall’s lemma, we conclude that $\|\mu_t^1 -\mu_t^2 \|_{TV} = 0$ for all $t\leq T$. Therefore, uniqueness holds in $[0,T]$.  
\end{proof}

We have now established all the necessary tools to prove Theorem \ref{thm: convemp}
\begin{proof}[Proof of theorem \ref{thm: convemp}]
	Due to Lemma \ref{lem: tighness} and by Prohorov's theorem \cite[Theorem 5.1]{billing}, every sequence has a sub-sequence convergent in law in $D\left([0,T], M^v_F(\mathbb{R})\right)$. We denote also by $\big\{\mu^K\big\}_K$ the convergent subsequence (for notational convenience), and suppose it converges in law to $(\mu_t)_{t\leq T}$ in the space $D\left([0,T], M^v_F(\mathbb{R})\right)$. For any $\nu \in D\left([0,T], M_F^v(\mathbb{R})\right)$, $\phi\in C_b^{1,2}([0,T]\times\mathbb{R})$ and $K>0$, we can define the function
	\begin{equation}\label{hl.26}
		\begin{split}
			\Phi^{1,\phi}_t(\nu) :=& 	\left< \phi, \nu_t\right> - \left< \phi, \nu_0\right>
			-\int_{0}^{t}\int_{\mathbb{R}}  \Big(\partial_t\phi(s, x) +\frac{1}{2} \partial_{xx}\phi(s, x ) \Big) \nu_s(dx)ds   \\
			&-\int_0^t \int_{\mathbb{R}}\phi(s,x)\nu_s(dx)ds  + \int_0^t \int_{\mathbb{R}} 2\int_x^\infty\nu_s(dy)  \phi(s,x)\nu_s(dx)ds \\
		\end{split},
	\end{equation}
	and for the subsequence $\{\mu^K\}_K$, we define 
	\begin{equation}\label{eq: phi2}
		\Phi^{2,\phi,K}_{t}(\mu^K) := -2c_K\int_0^t \int_{\mathbb{R}}   \phi(s,x)\mu_s^K(dx)ds - L_t^{K,\phi}.
	\end{equation}
	From Lemma \ref{lem: martngl}, we know that $\Phi^{1,\phi}_t(\mu^K) + \Phi^{2,\phi,K}_{t}(\mu^K) =0$ for every $K \in \mathbb{N}$, and also that the martingale $L_t^{K,\phi}$ vanish as $K\to \infty$ in mean. In turn, for the integral term on the right hand side of \eqref{eq: phi2} we have  
	\[\mathbb{E}\Bigg(\Bigg|2c_K\int_0^t \int_{\mathbb{R}}   \phi(s,x)\mu_s^K(dx)ds\Bigg|\Bigg) \leq 2c_K\|\phi\|_{\infty}\sup_K \Big[K^{-1}\mathbb{E}\Big(\sup_{t\leq T}(N_t^K)\Big)\Big],\]
	and due to Lemma \ref{lem: vagueapprox} we ensure it also vanish as $K\to \infty$.
	Therefore we have
	\[\lim_{K\to\infty}\mathbb{E}\Big(\Big|\Phi^{2,\phi,K}_{t}(\mu^K)\Big|\Big) =0. \]
	On the other hand, for any $\nu \in D\left([0,T], M_F^v(\mathbb{R})\right)$, there exists a constant $C(t,\phi)$, which is an increasing function of $t$, such that
	\begin{equation*}
		\big|\Phi^{1,\phi}_t(\nu)\big| \leq  C(T,\phi) \sup_{s\in [0,T]} \Big(\big<1,\nu_s\big> + \big<1,\nu_s\big>^2\Big).
	\end{equation*}
	Therefore
	\[\begin{split}
		\mathbb{E}\Big(\big|\Phi^{1,\phi}_t(\mu^K)\big|^{3/2}\Big) & \leq C^{3/2}(T,\phi) \mathbb{E}\Bigg(\bigg[\sup_{s\in [0,T]} \Big(\big<1,\mu^K_s\big> + \big<1,\mu^K_s\big>^2\Big) \bigg]^{3/2}\Bigg).
	\end{split}\]
	By taking the supremum in $K$, and from the fact that for $x>0$ the function $x\to x^{3/2}$ is increasing and continuous we obtain
	\[
	\sup_K\mathbb{E}\Big(\big|\Phi^{1,\phi}_t(\mu^K)\big|^{3/2}\Big) \leq \widetilde{C}^{3/2}(T,\phi) \sup_K\bigg[K^{-3}\mathbb{E}\bigg(\sup_{s\in [0,T]}  (N_s^K)^3\bigg)\bigg].
	\]
	In  Lemma \ref{lem: vagueapprox} we showed that the bound above is finite, and therefore the sequence $\{\Phi^{1,\phi}_t(\mu^K)\}_K$ is uniformly integrable. We know also that the function $\Phi^{1,\phi}_t(\cdot)$ is continuous. Therefore, we have
	\[\lim_{K\to\infty}\mathbb{E}\Big(\big|\Phi^{1,\phi}_t(\mu^K)\big|\Big)  =\mathbb{E}\Big(\lim_{K\to\infty}\big|\Phi^{1,\phi}_t(\mu^K)\big|\Big) = \mathbb{E}\Big(\big|\Phi^{1,\phi}_t(\mu)\big|\Big) .\]
	From this, we deduce that
	\[\lim_{K\to\infty}\mathbb{E}\Big(\big|\Phi^{1,\phi}_t(\mu^K) + \Phi^{2,\phi,K}_{t}(\mu^K)\big|\Big)  =\mathbb{E}\Big(\big|\Phi^{1,\phi}_t(\mu)\big|\Big). \]
	Since $\mathbb{E}\Big(\big|\Phi^{1,\phi}_t(\mu^K) + \Phi^{2,\phi,K}_{t}(\mu^K)\big|\Big)= 0$ for every $K \in \mathbb{N}$, it follows that
	$\mathbb{E}\Big(\big|\Phi^{1,\phi}_t(\mu)\big|\Big)=0$ almost surely. Consequently, $\Phi^{1,\phi}_t(\mu)=0$ and hence $\mu$ satisfies the equation \eqref{eq: eqmu0}.

	To extend this to the space $D\left([0,T], M_F(\mathbb{R})\right)$ (see Notation \ref{not}),  we use the following result that was first proved in \cite{meleard1993convergences} and is stated here as presented in \cite{Fontbona2013}. Namely, in addition to the convergence in $D\left([0,T], M_F^v(\mathbb{R})\right)$, the convergence in the space $D\left([0,T], M_F(\mathbb{R})\right)$ is ensured if the following two conditions holds:
	\begin{enumerate}
		\item[i)]  the process $(\mu_t)_{t\in[0,T]}$ is continuous in  $D\left([0,T], M_F(\mathbb{R})\right)$ ,
		\item[ii)]  the sequence $\big\{\big(\big<1,\mu^{K}_t\big>\big)_{t\in[0,T]} \big\}_{K\in\mathbb{N}}$ converges in law to the process $\big(\big<1,\mu_t\big>\big)_{t\in[0,T]} $ in the space $D([0,T],\mathbb{R})$. 
	\end{enumerate}
	We start with condition i).
	For any $f\in C_0(\mathbb{R})$ and $K\in\mathbb{N}$, consider the process $\big(\big<f,\mu_t^K\big>\big)_{t\in[0,T]}$ and the function $J: D([0,T],\mathbb{R})\to \mathbb{R}$  such that
	\[J\Big(\big<f,\mu^K\big>\Big):=\int_0^T e^{-u}\bigg(\sup_{t\leq u}\left|\left<f,\mu_t^K\right> - \left<f,\mu_{t-}^K\right>\right|\wedge 1\bigg)du\]
	Given that we weighted each particle by $1/m_K$, every jump will be of order $c/K$ and then
	\[\sup_{t\leq T}\Big|\left<f,\mu_t^K\right> - \left<f,\mu_{t-}^K\right>\Big|\leq \frac{c\|f\|_{\infty}}{K}\]
	Therefore, $J\big(\left<f,\mu^K\right>\big)\leq cK^{-1}T\|f\|_{\infty}$ which vanish as $K\to \infty$. By \cite[Theorem $10.2$]{Ethier2005}  we can say that $\left<f,\mu\right>\in C([0,T],\mathbb{R})$, i.e. 
	\begin{equation}\label{eq : vague continuity}
		\sup_{t\leq T}|\left<f,\mu_t\right> - \left<f,\mu_{t-}\right>|=0 \qquad \text{a.s.}, \quad f\in C_0(\mathbb{R}). 
	\end{equation}
	To extend the continuity to the space $D\big([0,T], M_F(\mathbb{R})\big)$, consider now a function $g\in C_b(\mathbb{R})$, the sequence $\{\psi_n\}_{n\geq 0}$ as defined in Lemma \ref{lem: vagueapprox}, and also the functions $g(1-\psi_n)\in C_0(\mathbb{R})$ for a sufficiently large $n$.  We have then
	\begin{equation*}
		\begin{split}
			|\left<g,\mu_t\right>-\left<g,\mu_{t-}\right>| &\leq |\left<g,\mu_t\right>-\left<g(1-\psi_n),\mu_{t}\right>|+ |\left<g(1-\psi_n),\mu_t\right>-\left<g(1-\psi_n),\mu_{t-}\right>|  \\
			&+ |\left<g(1-\psi_n),\mu_{t-}\right>-\left<g,\mu_{t-}\right>| \\
			& \leq \|g\|_{\infty}\left<\psi_n,\mu_t\right>+ |\left<g(1-\psi_n),\mu_t\right>-\left<g(1-\psi_n),\mu_{t-}\right>|+ \|g\|_{\infty}\left<\psi_n,\mu_{t-}\right>
		\end{split}.
	\end{equation*}
	By \eqref{eq : vague continuity}, we have that the second term is zero as $n$ goes to infinity. Also,  we know the first and third term go to zero as $n\to\infty$ due to Lemma \ref{lem: vagueapprox}, therefore we can conclude that 
	\[
	\sup_{t\leq T} |\left<g,\mu_t\right>-\left<g,\mu_{t-}\right>| =0, \qquad \text{a.s.},\quad g\in C_b(\mathbb{R}).
	\]
	To stablish condition  ii), consider a bounded $C$-Lipschitz continuous function $F:\mathbb{D}([0,T],\mathbb{R})\to\mathbb{R}$. For every $K,n\geq 1$ we have
	\begin{equation}
		\begin{split}
			\left|\mathbb{E}\left(F\left(\left<1,\mu^K\right>\right)\right)  -\mathbb{E}\left(F\left(\left<1,\mu\right>\right)\right)\right|  & \leq  \left|\mathbb{E}\left(F\left(\left<1,\mu^K\right>\right)\right)  -\mathbb{E}\left(F\left(\left<1-\psi_n,\mu^K\right>\right)\right)\right| \\
			&+ \left|\mathbb{E}\left(F\left(\left<1-\psi_n,\mu^K\right>\right)\right)  -\mathbb{E}\left(F\left(\left<1-\psi_n,\mu\right>\right)\right)\right|\\
			&+ \left|\mathbb{E}\left(F\left(\left<1-\psi_n,\mu\right>\right)\right)  -\mathbb{E}\left(F\left(\left<1,\mu\right>\right)\right)\right|\\
			&\leq  C\sup_K\mathbb{E}\left(\sup_{t\leq T}\left<\psi_n,\mu^K\right>\right) \\
			&+ \left|\mathbb{E}\left(F\left(\left<1-\psi_n,\mu^K\right>\right)\right)  -\mathbb{E}\left(F\left(\left<1-\psi_n,\mu\right>\right)\right)\right|\\
			&+ C\mathbb{E}\left(\sup_{t\leq T}\left<\psi_n,\mu\right>\right)
		\end{split}.
	\end{equation}
	Since $1-
	\psi_n\in C_0(\mathbb{R})$, and the functions $\nu\to \left<f,\nu\right>$ are continuous, the second term vanishes as $K\to \infty$, by vague convergence. For the first and third terms, Lemma \ref{lem: vagueapprox} ensures that both vanish in the limits $n\to \infty$ and $K\to \infty$. So we can conclude
	\[\lim_{K\to\infty} \mathbb{E}\left(F\left(\left<1,\mu^K\right>\right)\right) = \mathbb{E}\left(F\left(\left<1,\mu\right>\right)\right). \]
	
	Having established the weak convergence of the sequence $\{\mu^K\}_K$ to the limiting probability measure $\mu$, we now investigate whether 
	$\mu$ is absolutely continuous w.r.t. the Lebesgue measure. It is known that for each time $t>0$, the transition semigroup $P_t(x,dy)$ (of the Brownian motion) has, for each $x$, a unique density function $p_t(x,y)$ w.r.t. Lebesgue measure. This implies that, for every $f\in C_b(\mathbb{R})$,
	\begin{equation}\label{eq: semigroup}
		P_t f(x) = \int_{\mathbb{R}}f(y)p_t(x,y)dy    .
	\end{equation}
	Assume $\mu_t$ is a solution of \eqref{eq: mild} such that $\mu_0$ has a density w.r.t. the Lebesgue measure. Using the representation in \eqref{eq: semigroup}, we can write the equation \eqref{eq: mild} as follows
	\begin{equation*}
		\begin{split}
			\int_{\mathbb{R}}f(x) \mu_t(dx) = &\int_{\mathbb{R}} \left(\int_{\mathbb{R}}f(y)p_t(x,y)dy\right) \mu_0(dx) \\&+ \int_0^t \int_{\mathbb{R}} \left(\int_{\mathbb{R}}f(y)p_{t-s}(x,y)dy\right)\left(1  - 2  \int_x^{\infty} \mu_s(dy)\right)\mu_s(dx)ds
		\end{split} .   
	\end{equation*}
	Using Fubini's theorem we can ensure there exists a function $u \in L^{\infty}([0,T],L^1(\mathbb{R}))$ that satisfies 
	\[\int_{\mathbb{R}}f(x)\mu_t(dx) = \int_{\mathbb{R}}f(y)u(t,y)dy,\]
	and therefore, by \cite[Theorem $1.2$]{billing}, $\mu_t(dx) = u(t,x)dx$, so we can conclude $\mu_t$ is absolutely continuous  w.r.t. the Lebesgue measure.
\end{proof}

\section{Weak selection principle for Logistic branching Brownian motion}\label{microweak}
In order to prove Theorem \ref{thm: speed}, we show the existence of the limiting velocity for each competition rate $c>0$ by studying the successive return times of the population size to one. These return times define i.i.d. segments of the trajectory, and therefore we use the law of large numbers to prove the existence of an asymptotic velocity. 
Finally, we establish the lower bound of the limit in item ii) of Theorem \ref{thm: speed} by studying the stationary measure of the process $(\boldsymbol{X}^c(t))_{t\geq 0}$ as seen from the leftmost particle. 
The upper bound follows from the monotonous coupling between the Log-BBM and the BBM described in Subsection \ref{subsec: coupling}.
\subsection{The Log-BBM as seen from its minimum}\label{subsec: seen-min}
We will begin by defining the process we are interested in. For each time $t$, we consider the process $\{X^c_{(i)}(t), i=1,\dots,N^c_t\}$  of the order statistics of $\boldsymbol{X}^c(t)$, defined recursively as follows. 
The first order statistic is given by 
\[
X_{(1)}^c(t) = \min_{u \in \mathcal{N}_t^c} X^c_u(t),
\] 
which correspond to the particle indexed by $u_1 = \arg\min_{u \in \mathcal{N}_t^c} X^c_u(t)$. In general, given the $i$-th order statistic, we define the $(i+1)$-th as 
\[
X^c_{(i+1)}(t) = \min_{u \in \mathcal{N}_t^c\setminus\{u_1, \cdots, u_{i}\}} X^c_u(t),
\] 
which correspond to the particle indexed by $u_{i+1} = \arg\min_{u \in \mathcal{N}_t^c\setminus\{u_1, \cdots, u_{i}\}} X^c_u(t)$, and it is clear that the last order statistic $X^c_{N^c_t}(t)$  correspond to the  $\max_{u \in \mathcal{N}_t^c} X^c_u(t)$. Hence, we define the {\bf  Logistic-BBM process as seen from the minimum} by 
\begin{equation*}
	\boldsymbol{Z}^c(t) := \{Z^c_i(t) = X^c_{(i)}(t) - X^c_{(1)}(t), 1 \leq i \leq N_t^c\}, \qquad t\geq 0.
\end{equation*} 
The process $(\boldsymbol{Z}^c(t))_{t\geq 0}$ is also a Markov process, with state space $(\mathbb{E},\mathcal{B}(\mathbb{E}))$, where 
\[
\mathbb{E} := \bigcup_{n\in \mathbb{N}} \big\{ \{n\}\times\{(0, z_2, \cdots, z_n ) \in \mathbb{R}^n: 0\leq z_2\leq \cdots \leq z_n \}\big\}
\]
and $\mathcal{B}(\mathbb{E})$ is the Borel $\sigma$-field generated on $\mathbb{E}$. Note that, endowed with the disjoint union topology, the space $\mathbb{E}$ is a locally compact separable metric space. The following result ensures that $\boldsymbol{Z}$ admits a limit distribution. 
\begin{proposition}\label{prop: harris}
	The process $(\boldsymbol{Z}^c(t))_{t\geq 0}$ has an unique stationary distribution $\boldsymbol{\nu}^c$ on $\mathbb{E}$, such that 
	\[\boldsymbol{\nu}^c(dx_1,\cdots,dx_n,n) = \boldsymbol{\nu}^{c,n}(dx_1,\cdots,dx_n)\pi(n),\]
	such that $\int_{\mathbb{R}^n}  \boldsymbol{\nu}^{c,n}(dx_1,\cdots,dx_n)=1$ for each $n\in \mathbb{N}$. In addition, for any initial distribution $\boldsymbol{\psi}_{\pi}$,
	\begin{equation}
		\lim_{t \rightarrow\infty} \| P_{\boldsymbol{\psi}_{\pi}} \left(\boldsymbol{Z}^c(t) \in \cdot\right) - \boldsymbol{\nu}^c(\cdot)\|_{TV} =0.
	\end{equation}
\end{proposition}

\begin{proof}[Proof of Proposition \ref{prop: harris}]    To stablish this result, we will need some notation. Let the Markov chain $\{\mathbf{Z}^c_k\}_{k\in \mathbb{Z}_+}$ be the process obtained by sampling $(\mathbf{Z}^c(t))_{t\geq 0}$ at times $\{kT\}_{ k\in \mathbb{Z}_+}$. Denote by $P^t(\boldsymbol{\xi},A)$ the transition kernel of $(\mathbf{Z}^c(t))_{t\geq 0}$ for a configuration $\boldsymbol{\xi}\in\mathbb{E}$. Then, the sampled chain $\{\mathbf{Z}^c_k\}_{k\in \mathbb{Z}_+}$ has transition kernel $ K(\boldsymbol{\xi},A):= P^T(\boldsymbol{\xi},A).$ For this sampled chain, a set $R\in \mathcal{B}(\mathbb{E})$ is said to be small if there exist $r\in \mathbb{N}$ and a non-trivial measure $\phi_r$ such that  
	\begin{equation}\label{eq: smallset}
		K^r(\boldsymbol{\xi},A) \geq  \phi_r(A), \quad \forall \boldsymbol{\xi} \in R, \quad \ \text{and} \ \forall A\in\mathcal{B}(\mathbb{E})    
	\end{equation}
	with $K^r(\boldsymbol{\xi}, A) := P^{rT}(\boldsymbol{\xi}, A)$. When condition \eqref{eq: smallset} holds we say that $R$ is a $\phi_r$-small set. Set now
	\begin{equation}\label{eq: R}
		m:=\lceil\mathbb{E}(N_0^c)\rceil \quad \text{ and } \quad R := \{ \boldsymbol{\xi} \in \mathbb{E}: |\boldsymbol{\xi}| = m\ \text{and}\ \xi_{i+1} -\xi_i \in [0,m] \quad \forall i=1,\cdots, m-1\}     
	\end{equation}
	
	We begin by outlining the main steps of the proof.
	\begin{itemize}
		\item[i.]    We prove that there exists a non‑trivial measure $\phi_1$ such that the (closed) set $R \in \mathcal{B}(\mathbb{E})$ is $\phi_1-$small. Moreover,
		we show that \[\mathbb{P}_{\boldsymbol{\xi}}(\tau_R < \infty) =1, \quad \ \text{for all}\ \boldsymbol{\xi} \in \mathbb{E},\]
		where $\tau_R = \inf \{ k\geq 1: \mathbf{Z}^c_k \in R\}$. According to \cite[Theorem 1.1]{Meyn1993} these conditions ensure that the Markov chain $\{\boldsymbol{Z}^c_k\}_{k\in \mathbb{Z}_+}$ is Harris recurrent. Furthermore, by \cite[Theorem 3.1]{Meyn1993}, the existence of a Harris‑recurrent skeleton is sufficient for the Harris recurrence of the continuous‑time process. As a consequence, we have that there exists a unique invariant measure $\widetilde{\boldsymbol{\nu}}^c$.\label{en:i}
		\item[ii.] Next, we prove that if  we denote by $\tau_R(s)$ the first return time to $R$ after $s$, for some $s > 0$, then 
		\begin{equation}\label{eq: positiverecurr}
			\sup_{\boldsymbol{\xi} \in R} \mathbb{E}_{\boldsymbol{\xi}}(\tau_R(s)) < \infty. 
		\end{equation}
		Due to \cite[Theorem 1.2]{Meyn1993}, condition \eqref{eq: positiverecurr} ensures that the process $({Z}^c(t))_{t\geq 0}$ is positive Harris recurrent. Consequently, the invariant measure $\widetilde{\boldsymbol{\nu}}^c$ is finite and can therefore be normalized to obtain a probability measure $\boldsymbol{\nu}^c$. \label{en:ii} 
		\item[iii.] In order to prove ergodicity of the process $(\mathbf{Z}^c_t)_{t\geq 0}$, it is sufficient to establish the ergodicity of one of its skeleton chains \cite[Implication 20.25 $\Rightarrow$ 20.26]{Meyn1993c}. Therefore, we focus on the sampled chain $(\boldsymbol{Z}^c_k)_{k\in\mathbb{Z_+}}$ defined previously. For this chain to be ergodic, it must satisfy, in addition to positive recurrence, that it is also strongly aperiodic \cite[Theorem 13.3.1]{Meyn1993c}. We have already proved in Step~\ref{en:i} the existence of a $\phi_1$-small set $R$ such that $\phi_1(R)>0$, and by \cite[Subsection 5.4.3]{Meyn1993c} this is sufficient to establish strong aperiodicity.
		\label{en:iii}  
	\end{itemize}
	
	We now focus on the first part of Step (i).
	We want to prove that $R$ is indeed a small set for the sampled chain $\{\mathbf{Z}^c_k\}_{k\in\mathbb{Z}_+}$. If $A \subset  R$, we have that $K^1(\boldsymbol{\xi},A)$ is bounded from below by the probability that there is no branching nor competition before $T$ and in this time the process $\mathbf{Z}^c$ moves from $\boldsymbol{\xi} \in R$ to $A$. In this case, if we take $\tilde{\boldsymbol{x}}(T):=\boldsymbol{\eta} + \boldsymbol{\widetilde{B}}_T$  where $\boldsymbol{\widetilde{B}}$ is an $m$-dimensional Brownian motion and $\eta\in \mathbb{R}^m$ such that $\xi_i = \eta_{(i)} - \eta_{(1)}$, and we set
	\[\begin{split}
		E_{A}:=\left\{\left(\tilde{\boldsymbol{x}}_{(1)}(T),\dots,\tilde{\boldsymbol{x}}_{(m)}(T)\right)\in A + \left(\tilde{\boldsymbol{x}}_{(1)}(T),\dots,\tilde{\boldsymbol{x}}_{(1)}(T)\right)\right\},   
	\end{split}
	\] we have that 
	\begin{equation} \label{eq: mu }
		K^1(\boldsymbol{\xi},A) \geq  e^{-\kappa_{m} T} \mathbb{P}(E_A),				
	\end{equation}
	where $\kappa_{m} =m( 1  + c(m-1))$.
	Let $f_T(y -x )$ be the transition density for a single
	one-dimensional Brownian motion run for time $T$, starting from $x \in \mathbb{R}$, then 
	\[\begin{split}
		\mathbb{P}(E_A) =\int\sum_{\sigma\in P_m}\prod_{j=1}^{m} \mathbb{1}_{\{y_i - y_1\in A\}}f_T(y_i -\eta_{\sigma(i)})\mathbb{1}_{\{y_1\leq y_2 \leq \dots \leq y_{m} \}}dy_1dy_2\dots dy_{m} 
	\end{split}\]
	where $P_m$ is the set of all permutations of $\{1,\cdots,m\}$. Making the change of variables $u_i = y_i - y_1$ and $z=y_1 - \eta_{(1)}$, and 
	using that $\xi_{\sigma(i)} =\eta_{\sigma(i)}-\eta_{(1)}$, we obtain 
	\[\begin{split}
		\mathbb{P}(E_A) =\sum_{\sigma\in P_m}\int_{-\infty}^{\infty}f_T(z -\xi_{\sigma(1)})\int_A\prod_{j=2}^{m} f_T(u_i +z -\xi_{\sigma(i)})\mathbb{1}_{\{0\leq u_2 \leq \dots \leq u_{m} \}}du_2\dots du_{m}dz
	\end{split}\]
	Furthermore, we know that on $A$, $u_i \leq m(i-1)\leq m(m-1)$,  and that the same bound holds for $\xi_{\sigma(i)}$ on $R$. Therefore, \[f_T(u_i + z -\xi_{\sigma(i)})\geq \inf_{|w|\leq 2m(m-1)}f_T(w +|z|)=f_T(2m(m-1) +|z| ).\]
	Hence,
	\[\begin{split}
		\mathbb{P}(E_A) &\geq \sum_{\sigma\in P_m}\int_{-\infty}^{\infty}f_T(z -\xi_{\sigma(1)})f_T(2m(m-1) +|z| )^{m-1}\int_A \mathbb{1}_{\{0\leq u_2 \leq \dots \leq u_{m} \}}du_2\dots du_{m}dz\\
		& \geq\lambda(A)\int_{-\infty}^{\infty}f_T(z )f_T(2m(m-1) +|z| )^{m-1}dz
	\end{split}\]
	where in the last inequality we use that the sum over all permutations is bounded below by any particular term; here we choose one such that $\xi_{\sigma(1)}=0$. Setting 
	\begin{equation}
		p =  \int_{-\infty}^{\infty}f_T(z )f_T(2m(m-1) +|z| )^{m-1}dz
	\end{equation}
	we can take $\phi_1(\cdot) = p \lambda(R \cap \cdot)$ with  $\lambda$ the Lebesgue measure. Therefore, $R$ is a $\phi_1$-small set.
	
	Having established the existence of the small set, we now turn to the second part of Step (i), that is, showing that the chain hits the set $R$ in finite time almost surely, regardless of the initial configuration. Starting from any initial configuration $\boldsymbol{\xi}_{N^c_0}$ (which means that $|\boldsymbol{\xi}| =N^c_0$ and $N^c_0 \sim \pi$) the process could get into $R$ by killing all the initial particles (but one) and from that point have at least $m-1$ births and no deaths and the corresponding movement of the particles does not leave $(-1/2,1/2)$ between the birth times of the particles.
	
	Denote now by $\boldsymbol{\xi}^n$ an initial configuration of the particles when $|\boldsymbol{\xi}|=n, n>1$. Each initial particle $u \in \mathcal{N}^c_0$ has $n-1$  i.i.d. competition times $\{J_{u,v}^c\}_{v \in \mathcal{N}_0 \setminus \{u\}}$, where $J_{u,v}^c \sim \text{Exp}(c)$; and a given  reproduction time $J_u^{b} \sim \text{Exp}(1)$. This reproduction time can be equivalently written as
	\[J_u^{b} = \inf\{t>0: \exists v\in \mathbb{T}^c\ \text{such that}\ v \ \text{is a child of} \ u \ \text{with}\ \tau_v^b=t \}\] 
	in agreement with the construction given in Subsection \ref{subsec: coupling}. Define the discrete time process $\{\widetilde{N}^c_k\}_{k\geq 1}$
	such that $\widetilde{N}^c_k := N^c_{kT}$ for $k\geq 0$ and \[\tau_{1} := \inf\{k>0: \widetilde{N}^c_k=1\}.\] 
	We want to find a lower bound for the probability of the event $\{\tau_1 = 1\}$. If we define for the process $\boldsymbol{Z}(t)$ the events 
	\begin{equation*}
		\begin{split}
			A_1 &:=\left\{\text{ there is no reproduction event before time $T$}\right\}
			,\\
			A_2&:=\left\{\text{there are exactly $n - 1$ competition events in the time interval $[0,T]$}\right\} ; 
		\end{split}
	\end{equation*}
	then it is clear that $A_1 \cap A_2 \subset\{\tau_1 = 1\}$. Note that $\mathbb{P}_{\boldsymbol{\xi}^n}(A_1) = \exp\{-T n\}$.
	
	On the other hand, let $S_1^c$ be the first time at which a competition event occurs, and let $\{S_k^c\}_{k \geq1}$ be the subsequent competition times. We write $A_2 = \{S_{n -  1}^c \leq T\}$, and hence
	\[\mathbb{P}_{\boldsymbol{\xi}^n}(\tau_1 = 1) \geq \mathbb{P}_{\boldsymbol{\xi}^n}\big(A_1\big)\mathbb{P}_{\boldsymbol{\xi}^n}\big(S_{n -  1}^c \leq T \big|A_1\big).\] 
	For $ i \geq 1$, define the inter-competition times by $\Delta^c_i := S_{i}^c - S_{i-1}^c$. Conditional on $A_1$, the random variables $\{\Delta^c_i\}_{i\geq 1}$ are independent, with  $\Delta^c_i \sim \text{Exp}(c(n-i+1))(n-i))$. In particular, $S_{n-1}^c =\sum_{k=1}^{n -  1}\Delta^c_k$, and we have that
	\begin{equation}
		\begin{split}
			\mathbb{P}_{\boldsymbol{\xi}^n}\big(S_{n -  1}^c \leq T \big|A_1\big) 
			\geq \mathbb{P}_{\boldsymbol{\xi}^n} \Bigg(\max\limits_{k=1,\dots,n-1} \Delta^c_k \leq \frac{T}{n-1}\Bigg| A_1 \Bigg)
			= \exp\Bigg\{- \dfrac{cTn(n+1)}{3}  \Bigg\} 
		\end{split}	
	\end{equation} 
	Therefore, this allows us to conclude that  
	\begin{equation}
		\begin{split}
			\mathbb{P}_{\boldsymbol{\xi}_{N^c_0}}(\tau_1 = 1)  \geq  \frac{1}{1-e^{-1/c}}\sum_{n=1}^{\infty} \frac{\Bigg(\dfrac{1}{c} \exp \Bigg\{-\Bigg(1 + \dfrac{c(n-1)}{3}\Bigg)T \Bigg\} \Bigg)^n}{n!}=: C_1(T, c).
		\end{split}
	\end{equation} 
	with $C_1(T, c)\in (0,1)$. From this point on, we want to know if there exists a strictly positive lower bound for the probability $\mathbb{P}_{(0)}(\tau_R = 1)$, where $\mathbb{P}_{(0)}$ denotes the law of the process started with a single particle at  $0$. To this end, let $S_1^b$ denote the reproduction time of the remaining particle at time $\tau_1$, and let $\{S_k^{b}\}_{k \geq1}$ be the subsequent reproduction times. Define the inter‑reproduction times by $\Delta^{b}_i := S_{i}^{b} - S_{i-1}^{b}$ for $ i \geq 1$. Furthermore, for particles $u_i$ and $u_j$ born after time $\tau_1$, let 
	$J_{u_i,u_j}^c \sim \text{Exp}(c)$ denote their competition time. We are interested in the events for the process $\boldsymbol{Z}(t)$,  
	\begin{equation*}
		\begin{split}
			A_3 &:=\left\{\text{there are exactly $m-1$ branching events  before time $T$} \right\}=\left\{ S_{m-1}^{b} \leq T <S_{m-1}^{b}+ \Delta_m^{b}\right\}
			,\\
			A_4&:=\left\{\text{there are no competition events in the time interval $[0,T]$}\right\} = \left\{\min_{i,j\in [ 1, m], i \neq j}  J^{c}_{u_i, u_j} > T\right\}\\
			A_5 &:=\left\{\text{particles does not displace more than $1/2$ between the birth times of each new particle}\right\}.
		\end{split}
	\end{equation*}
	Using the variables previously defined, $A_5$ can be represented by
	\[
	A_5= \Bigg\{\bigcap_{\stackrel{i\in [ 2, m-1]}{r\in [ 1, i]}} B^r(\Delta^{b}_{i}) \in\Big(-\frac{1}{2},\frac{1}{2}\Big)\Bigg\} 
	\bigcap \Bigg\{\bigcap\limits_{q\in [ 1, m]}  B^q(T-S_{m-1}^{b}) \in\Big(-\frac{1}{2},\frac{1}{2}\Big)\Bigg\}, 
	\]
	where all $B^{q}, B^{r}$ are standard Brownian motions, independent of each other. It is clear that $A_3 \cap A_4 \cap A_5 \subset\{\tau_R = 1\} $. Then, we have that
	\begin{equation}
		\begin{split}
			\mathbb{P}_{(0)}(\tau_R = 1) \geq  \mathbb{P}(A_3,A_4,A_5)=\mathbb{P}(A_4)\mathbb{P}\Big( A_3, A_5 \Big | A_4\Big).
		\end{split}		
	\end{equation}
	We know that
	$\mathbb{P}(A_4) = e^{-Tcm(m-1)}$ and that, conditional on $A_4$, the random variables $\{\Delta^{b}_i\}_{i\geq 1}$ are independent with $ \Delta^{b}_i \sim \text{Exp}(i)$. Let $f_i$ denote the density of $\Delta^{b}_{i}$ for $i \in [ 1, m]$, conditional on  $A_4$. Then, for the last term on the right‑hand side of the inequality, we have
	\begin{equation*}
		\begin{split}
			\mathbb{P}\Big( A_3,A_5 \Big | A_4\Big)
			= &\int_0^\infty \cdots \int_0^\infty \mathbb{1}_{\left\{\sum_{i=1}^{m-1} z_i \leq T \leq \sum_{i=1}^{m} z_i\right\}}\prod_{i=2}^{m-1}\mathbb{P}\Big( B(z_i)\in\Big(-\frac{1}{2},\frac{1}{2}\Big)\Big)^i\\
			&\mathbb{P}\Big(B\big(T-\sum_{i=1}^{m-1} z_i\big) \in\Big(-\frac{1}{2},\frac{1}{2}\Big)\Big)^m f_1(z_1)f_2(z_2)\cdots f_m(z_m) dz_1 \cdots dz_m 
		\end{split}		
	\end{equation*} 
	By standard analysis arguments we can prove that there exists a constant $C_2(T,c) \in (0,1)$ such that bounds from below the probability $\mathbb{P}_{(0)}(\tau_R = 1)$. By the Strong Markov property
	\begin{equation*}
		\mathbb{P}_{\boldsymbol{\xi}_{N^c_0}}(\tau_R = 2) = \mathbb{P}_{\boldsymbol{\xi}_{N^c_0}}(\tau_1 = 1) \mathbb{P}_{(0)}(\tau_R = 1)  \geq C_1(T, c)C_2(T, c).
	\end{equation*}
	If we denote by $C_3(T, c) = C_1(T, c)C_2(T, c)$ then $\mathbb{P}_{\boldsymbol{\xi}_{N^c_0}}(\tau_R > 2) \leq 1- C_3(T, c)$ and using the strong Markov property again we deduce that 
	\begin{equation*}
		\begin{split}
			\mathbb{P}_{\boldsymbol{\xi}_{N^c_0}}(\tau_R > 2n)\leq \left(1-C_3(T, c)\right)^n
		\end{split}
	\end{equation*}
	Let $E_n=\{ \tau_R > 2n\}$ 
	\begin{equation}
		\sum_{n=1}^\infty \mathbb{P}_{\boldsymbol{\xi}_{N^c_0}}(E_n) \leq \sum_{n=1}^\infty \left(1-C_3(T, c)\right)^n <\infty
	\end{equation}
	therefore by Borel-Cantelli lemma $\mathbb{P}_{\boldsymbol{\xi}_{N^c_0}}(\limsup_n E_n) =0$, and  therefore $\mathbb{P}_{\boldsymbol{\xi}_{\pi}}(\tau_R = \infty)=0 $. So we can conclude that $\mathbb{P}_{\boldsymbol{\xi}_{\pi}}(\tau_R < \infty) =1$. 
	
	We now turn to Step (ii). We know that $\mathbb{P}_{\boldsymbol{\xi}_{N^c_0}}(\tau_R \geq 2n) \leq (1-C_3(c))^n$ and  the set $R$ is closed  on $(\mathbb{E},\mathcal{B}(\mathbb{E}))$ with the disjoint union topology. Using the tail integral representation of the mean, and given that on $[2(n-1),2n)$,  $\mathbb{P}_{\boldsymbol{\xi}_{N^c_0}}(\tau_R \geq x ) \leq \mathbb{P}_{\boldsymbol{\xi}_{N^c_0}}(\tau_R \geq 2(n-1))$, we have 
	\begin{equation}
		\mathbb{E}_{\boldsymbol{\xi}_{N^c_0}}(\tau_R) \leq 	  \sum_{n=1}^{\infty} \int_{0}^{\infty} \mathbb{P}_{\boldsymbol{\xi}_{N^c_0}}(\tau_R \geq 2(n-1) ) \mathbb{1}_{\{x\in[2(n-1),2n) \}}dx =  \sum_{n=1}^{\infty}  \left(1-C_3(c)\right)^{n-1} ,
	\end{equation}
	which converges. As a consequence we can secure that 
	\begin{equation}
		\sup_{\boldsymbol{\xi}_{N^c_0}\in \mathbb{E}} \mathbb{E}_{\boldsymbol{\xi}_{N^c_0}}(\tau_R)  	   <\infty, \ \text{and therefore}\ \sup_{\boldsymbol{\xi}\in R} \mathbb{E}_{\boldsymbol{\xi}}(\tau_R(1))  	   <\infty.
	\end{equation} 
\end{proof}

\subsection{Existence of the limiting velocity}\label{subsec: existence}
In this subsection, we prove that the asymptotic velocity of the rightmost particle exists and is finite and non-random. 
To this end, we fix $c > 0$ and let $(N_t^c)_{t\geq 0}$ be the LB-process associated to our spatial dynamic. Following the approach of \cite{Pain2016}, our strategy is to apply the law of large numbers to the renewal structure induced by the successive return times of $N^c_t$ to one. In order to apply the law of large numbers, we first need to show that, between return times, the maximum of $(\boldsymbol{X}^c(t))_{t\geq 0}$ does not grow too much in expectation. 

\begin{lemma} \label{lem3}
	Let $(\boldsymbol{X}^c(t))_{t\geq 0}$ denote the Log-BBM($1,c$) process. Given $\varepsilon_1 \sim$ Exp$(1)$, define the times 
	\[
	T_1 := \inf \{t\geq 0: N^c_t = 1\} \quad \text{ and } \quad
	T_2 := \inf \{ t \geq T_1 + \varepsilon_1: N_t^c = 1 \},
	\]
	which are the hitting time of one and the first return time to one of the process $N_t^c$, respectively. Set $T:=T_2 - T_1$, and we define 
	\begin{equation*}
		\gamma := \max\limits_{t \in [0,T]}\Big|\max\limits_{u \in \mathcal{N}^c_t} X^c_u(t) \Big|.
	\end{equation*}
	Then $\mathbb{E}_{(0)}(\gamma)<\infty$, where $\mathbb{E}_{(0)}$ denote expectation under the law of the process starting from one particle at position zero. 
\end{lemma}

\noindent
The proof closely follows that of  \cite{Pain2016}, however, we include it in Appendix \ref{appendix: lem3} for completeness. Now we are ready to prove Theorem \ref{thm: speed} \ref{it: vmaxequalvmin}, that is the existence of the asymptotic speed of the cloud of particles of the Log-BBM.

\begin{proof}[Proof of Theorem \ref{thm: speed} \ref{it: vmaxequalvmin}]
	Consider the sequence $\{T_i\}_{i\in \mathbb{N}}$ such that 
	$T_1 := \inf \{t\geq 0: N^c_t = 1\}$ and  $T_{i+1} := \inf \{t\geq T_i + \varepsilon_i: N^c_t = 1\}$, with  i.i.d random variables $\varepsilon_i \sim \text{Exp}(1)$, $i\in \mathbb{N}$.  By the strong Markov property $\{T_{i+1}-T_{i}\}_{i\in \mathbb{N}}$ is a sequence of i.i.d. random variables with the same law as some random variable $T$. Define now the sequence $\{M_i\}_{i\in \mathbb{N}}$ such that 
	\[M_i := \max\limits_{u \in \mathcal{N}^c_{T_{i+1}}} X^c_u\left(T_{i+1}\right) -\max\limits_{u \in \mathcal{N}^c_{T_{i}}} X^c_u\left(T_{i}\right), \quad i \in \mathbb{N}.\]
	Note that at each time $T_i$ there is just one particle $u^*$ in the system, so by the strong Markov property  $\{M_i\}_{i\in \mathbb{N}}$ is a sequence of i.i.d. random variables such that, under $\mathbb{P}_{(0)}$
	\[M_i \stackrel{\text{d}}{=}  \max\limits_{u \in \mathcal{N}^c_{T_{2}}} X^c_u\left(T_{2}\right) -\max\limits_{u \in \mathcal{N}^c_{T_1}} X^c_u\left(T_1\right)= X^c_{u^*}\left(T_2\right)-X^c_{u^*}\left(T_1\right)
	\stackrel{\text{d}}{=} X^c_{u^*}\left(T\right)=\max\limits_{u \in \mathcal{N}^c_T} X^c_u\left(T\right) \quad i\in \mathbb{N} \ 
	\] 
	Recall that the underlying LB-process is positive recurrent, so $\mathbb{E}_{\xi_{N_0^c}}\Big(T_{i+1} - T_{i}\Big )=\mathbb{E}_{(0)}\Big(T\Big )$ is finite. On the other hand, by lemma \ref{lem3} we have that $\mathbb{E}_{\xi_{N_0^c}}\Big(M_i\Big )=\mathbb{E}_{(0)}\Big(\max\limits_{u \in \mathcal{N}^c_{T}} X^c_u(T)\Big )<\infty$, so we can use the Law of Large Numbers to obtain that  
	\begin{equation*}
		\lim_{n\rightarrow\infty} \frac{T_{n} -T_1}{n}  =    \lim_{n\rightarrow\infty}\sum_{i=1}^{n} \frac{T_{i+1}-T_{i}}{n} =  \mathbb{E}_{(0)}(T) \quad  \mathbb{P}_{\xi_{N_0^c}}\text{-a.s. } \text{ and in } L^1,
	\end{equation*} 
	and
	\[
	\lim_{n\rightarrow\infty}  \frac{ \max\limits_{u \in \mathcal{N}^c_{T_n}} X^c_u\big(T_n\big) - \max\limits_{u \in \mathcal{N}^c_{T_1}} X^c_u\big(T_1\big)}{n}=
	\lim_{n\rightarrow\infty}\sum_{i=1}^n \frac{M_i}{n} = \mathbb{E}_{(0)}\Big(\max\limits_{u \in \mathcal{N}^c_T} X^c_u(T)\Big ) \quad  \mathbb{P}_{\xi_{N_0^c}}\text{-a.s. } \text{ and in } L^1.\]
	Given that $T_1 $ and $\max\limits_{u \in \mathcal{N}^c_{T_1}} X^c_u(T_1) $ are finite $\mathbb{P}_{\xi_{N_0^c}}$-a.s. , we have that both limits holds $\mathbb{P}_{\xi_{N_0^c}}$-a.s. and in $L^1$, by Lemma \ref{lem3}. Thus, we have 
	\begin{equation}\label{eq: LGNest}
		\lim_{n\rightarrow\infty}  \frac{\max\limits_{u \in \mathcal{N}^c_{T_n}} X^c_u\left(T_n\right)}{T_n} =\frac{\mathbb{E}_{(0)}\Big(\max\limits_{u \in \mathcal{N}^c_T} X^c_u(T)\Big )}{\mathbb{E}_{(0)}(T)} =: v_c \quad \mathbb{P}_{\xi_{N_0^c}} \text{-a.s  and in}\ L^1
	\end{equation}
	We want to prove the limit 
	\begin{equation}\label{eq: LGNpdq}
		\lim\limits _{t\rightarrow\infty} \frac{\max\limits_{u \in \mathcal{N}^c_t} X^c_u(t)}{t} = \lim\limits _{t\rightarrow\infty} \frac{\max\limits_{u \in \mathcal{N}^c_t} X^c_u(t) - \max\limits_{u \in \mathcal{N}^c_{T_{k(t)}}} X^c_u\left(T_{k(t)}\right)}{t} + \frac{ \max\limits_{u \in \mathcal{N}^c_{T_{k(t)}}} X^c_u\left(T_{k(t)}\right)}{T_{k(t)}}\frac{T_{k(t)}}{t},   
	\end{equation}
	for $k(t) = \sum\limits_{k=0}^{\infty} k \mathbb{1}_{T_{k} \leq t < T_{k+1}}$, exists and is equal to $v_c$.  
	For the second term in the r.h.s. of the equation above, it is clear that 
	\[
	\lim\limits_{t\rightarrow\infty}  \frac{ \max\limits_{u \in \mathcal{N}^c_{T_{k(t)}}} X^c_u\left(T_{k(t)}\right)}{T_{k(t)}} =v_c < +\infty
	\] and   
	$T_{k(t)} \leq t < T_{k(t)+1}$ implies that  $T_{k(t)} /k(t) \leq t/k(t) < T_{k(t)+1} /k(t)$ and consequently \\ $\lim_{t\rightarrow\infty} t/k(t) = \mathbb{E}_{(0)}(T)$. Therefore, 
	\[
	\lim\limits_{t\rightarrow\infty} \frac{T_{k(t)}}{t} = \lim\limits_{t\rightarrow\infty} \frac{T_{k(t)}}{k(t)}\frac{k(t)}{t} =1,
	\] thus the second term in the r.h.s. of \eqref{eq: LGNpdq} in the  goes to $v_c$  $\mathbb{P}_{\xi_{N_0^c}}$-a.s.
	For the first term in the r.h.s. of \eqref{eq: LGNpdq}, we have that 
	\begin{equation}
		\max\limits_{u \in \mathcal{N}^c_{t}} X^c_u(t) - \max\limits_{u \in \mathcal{N}^c_{T_{k(t)}}} X^c_u\big(T_{k(t)}\big)\leq  \max_{t\in \big[T_{k(t)},T_{k(t)+1}\big)}\Bigg| \max\limits_{u \in \mathcal{N}^c_t} X^c_u(t) - \max\limits_{u \in \mathcal{N}^c_{T_{k(t)}}} X^c_u\big(T_{k(t)}\big)\Bigg|
	\end{equation}
	If we set 
	\[
	\gamma_k:=\max_{t\in \big[T^{k}_1,T^{k+1}_1\big)}\Bigg| \max\limits_{u \in \mathcal{N}^c_t} X^c_u(t) - \max\limits_{u \in \mathcal{N}^c_{T_{k}}} X^c_u\big(T_{k}\big)\Bigg|,
	\]
	the sequence $(\gamma_k)_{k\geq 1}$, is a sequence of independent and equally distributed  random variables such that $\gamma_k \stackrel{\text{d}}{=} \gamma$ under $\mathbb{P}_{(0)}$, then by Lemma \ref{lem3} we know that $\gamma_k$ is finite for each $k\geq 1$, so as $t \rightarrow \infty$, $\gamma_k/t \rightarrow 0$  $\mathbb{P}_{(0)}$-a.s, and therefore the first term in the r.h.s. of \eqref{eq: LGNpdq} goes to zero $\mathbb{P}_{\xi_{N_0^c}}$-a.s. as $t\rightarrow\infty$.
	Moreover, given that at times $T_k$ both $\max\limits_{u \in \mathcal{N}^c_{T_k}} X^c_u(T_k)$ and $\min\limits_{u \in \mathcal{N}^c_{T_k}} X^c_u(T_k)$ coincide we can conclude that
	\begin{equation*}
		\lim\limits _{t\rightarrow\infty} \frac{\max\limits_{u \in \mathcal{N}^c_t} X^c_u(t)}{t} = \lim\limits _{t\rightarrow\infty} \frac{\min\limits_{u \in \mathcal{N}^c_t} X^c_u(t)}{t} =v_c \quad \mathbb{P}_{\xi_{N_0^c}}\text{-a.s. and in}\ L^1
	\end{equation*}
\end{proof}

\subsection{Velocity Selection}\label{subsec: speedselection}

To prove that equation Theorem \ref{thm: speed} \ref{it: speed0} holds true, we will need two auxiliary results. The first concerns the evolution of the empirical measure of the Log-BBM, and the second deals with the gaps between the particles. First, we will establish some notation that we will need.

For each $K \in \mathbb{N}$, let  $\big(\widetilde{\mu}_t^K\big)_{t\geq 0}$ be the empirical measure (mass normalized) of the Log-BBM($1,c_K$) process, defined by
\begin{equation}\label{eq: massempmeas}
	\widetilde{\mu}_t^K: = \frac{1}{N^K_t}\sum_{u\in\mathcal{N}_t^K}\delta_{X_u^K(t)}, \qquad t\geq 0, 
\end{equation}
where $\mathbf{X}^K=(\mathbf{X}^K(t))_{t\geq 0}$ is the (rescaled) Log-BBM($1, c_K$).
We denote by $(\overline{\mathbf{X}}^K(t))_{t\geq 0}$ the empirical mean of $\mathbf{X}^K$   and by $\widetilde{F}^K(t,x)$ the (random) cumulative distribution function of $\widetilde{\mu}_t^K$, i.e.
\begin{equation}\label{eq: empmean}
	\overline{\mathbf{X}}^K(t):= \left<\text{Id},\widetilde{\mu}_t^K\right>,\quad\quad \widetilde{F}^K(t,x) := \widetilde{\mu}_t^K((-\infty,x]) \quad t\geq 0,\ x\in\mathbb{R}. 
\end{equation}
\begin{remark}\label{rem: equalmeas}
	Both empirical measures $\mu^K$ and  $\widetilde{\mu}_t^K$,  defined in \eqref{eq: defemp0} and \eqref{eq: massempmeas} respectively, have the same weak limit. Namely, let $\mu_t$ be the limiting measure of $\mu_t^K$, defined in Theorem \ref{thm: convemp}. Then for all $\epsilon>0$ and all  $f\in C_b^2(\mathbb{R})$ we have 
	\[\mathbb{P}\left[\left|\left< f, \widetilde{\mu}_t^K\right> - \left< f, \mu_t\right>\right|>\epsilon\right] \leq  \mathbb{P}\left[\left|\left< f, \widetilde{\mu}_t^K\right>-\left< f, \mu_t^K\right>\right|>\epsilon/2\right]+\mathbb{P}\left[\left|\left< f, \mu_t^K\right> - \left< f, \mu_t\right>\right|>\epsilon/2\right]\]
	We know that $\mathbb{P}\left[\left|\left< f, \mu_t^K\right> - \left< f, \mu_t\right>\right|>\epsilon/2\right]\to 0$ as $K\to \infty$, and using  Markov and Cauchy-Schwarz inequalities we obtain 
	\[ \begin{split}  
		\mathbb{P}\left[\left|\left< f, \widetilde{\mu}_t^K\right>-\left< f, \mu_t^K\right>\right|>\epsilon/2\right] \leq \frac{2\|f\|_{\infty}}{m_K\epsilon}\sqrt{\text{Var} \Big( N_0^K \Big)}= \frac{2\|f\|_{\infty}}{m_K\epsilon}\sqrt{m_K-m_K^2 e^{-K/c}}
	\end{split}\]
	which vanishes as $K\to \infty$. Hence, we can conclude that  the limit in probability $\widetilde{F}$ of $\widetilde{F}^K(t,x)$ exists and satisfies $\widetilde{F} = F$ a.s., where $F$ is defined in Corollary \ref{Col: cdf}.
\end{remark}
The quantity $\mathbf{\overline{X}}^K(t)$ can be expressed in terms of  $\widetilde{F}^K(t,x)$ by employing the integral representation of the positive and negative parts of each $X^K_u$, as shown below. 
\[\begin{split}
	\overline{\mathbf{X}}^K(t) &= \frac{1}{N^K_t} \sum_{u\in \mathcal{N}^K_t}\bigg( \int_0^\infty \mathbb{1}_{\{X^K_u(t)>x\}}dx - \int_{-\infty}^0 \mathbb{1}_{\{X^K_u(t)\leq x\}}dx\bigg)\\
	& = \int_0^\infty \big(1-\widetilde{F}^K(t,x)\big)dx - \int_{-\infty}^0 \widetilde{F}^K(t,x)dx.
\end{split}\]
Assuming any initial distribution $\psi_{\pi}$, we take expectation on both sides and then differentiate with respect to $t$ to obtain 
\begin{equation}\label{eq: dervempmean}
	\begin{split}
		\partial_t	\mathbb{E}_{\psi_{\pi}}\Big(\overline{\mathbf{X}}^K(t)\Big) &= -\int_{-\infty}^\infty \partial_t\mathbb{E}_{\psi_{\pi}}\Big(\widetilde{F}^K(t,x)\Big)dx \\
		&= \int_{-\infty}^\infty \mathbb{E}_{\psi_{\pi}}\bigg[\frac{c_K(N^K_t)^2}{N^K_t-1}\Big(\widetilde{F}^K(t,x) 
		- \widetilde{F}^K(t,x)^2\Big)\bigg] dx
	\end{split},
\end{equation}
thanks to the following Lemma, whose proof is added in Appendix \ref{appendix: empmsreq}, and the fact that the integral of $\partial_{xx} \mathbb{E}_{\psi_{\pi}}\big(\widetilde{F}^K(t,x)\big)$ vanish. 
\begin{lemma}\label{lem: empmsreq}
	For all $t\geq 0$, $x \in \mathbb{R}$, we have that    \begin{equation}\label{fini_eq}
		\partial_t\mathbb{E}_{\psi_{\pi}}\Big(\widetilde{F}^K(t,x)\Big) = \frac{1}{2}\partial_{xx} \mathbb{E}_{\psi_{\pi}}\Big(\widetilde{F}^K(t,x)\Big)	-\mathbb{E}_{\psi_{\pi}}\Bigg[\frac{c_K(N^K_t)^2}{N^K_t-1}\Big(\widetilde{F}^K(t,x) 
		- \widetilde{F}^K(t,x)^2\Big)\Bigg]
	\end{equation}
\end{lemma}

Before stating the next lemma, we clarify the concept of stochastic domination. We say that two probability measures $\mu_1$ and $\mu_2$ in $M_F(\mathbb{R})$ satisfy $\mu_1\leq_{st} \mu_2$ if, for every $x\in \mathbb{R}$, $\mu_1([x,\infty))\leq \mu_2([x,\infty))$. In this case, we say that $\mu_2$ stochastically dominates $\mu_1$. On the other hand, for random variables $X,Y \in \mathbb{R}$ we write $X\leq_{st} Y$ if the distribution of $X$ is stochastically dominated by the distribution of $Y$.
Now, we show that the gaps between the particles in the Log‑BBM evolve monotonically.
\begin{lemma}\label{gaps}
	Let $(\mathbf{X}^{1,n}(t), t\geq 0)$  and   $(\mathbf{X}^{2,n}(t), t\geq 0)$ be two copies of the Log-BBM$(1,c)$, starting from $n \geq 2$  particles with initial configurations $\mathbf{X}^{1,n}(0)$ and $\mathbf{X}^{2,n}(0)$, respectively. If the processes satisfies that
	\[X^{1,n}_{(i+1)}(0)-X^{1,n}_{(i)}(0)\leq_{st} X^{2,n}_{(i+1)}(0)-X^{2,n}_{(i)}(0) \quad \forall \,i\in[1,n-1],\]
	therefore, for all $t\geq 0$, 
	\[X^{1,n}_{(i+1)}(t)-X^{1,n}_{(i)}(t)\leq_{st} X^{2,n}_{(i+1)}(t)-X^{2,n}_{(i)}(t) \quad \forall \, i\in[1,N^n_t-1].\]
\end{lemma}
\begin{proof}
	As the processes has the same law and initial number of particles, we can construct a coupling such as they share the sequence of branching times $\left(\tau^{b,n}_k\right)_{k\geq 1}$ and competition times $\left(\tau^{d,n}_k\right)_{k\geq 1}$, and also are driven by the same standard Brownian motions between such events. To each $\tau^{b,n}_k$ we associate a uniform random variable $U^n_k\sim \text{Unif}([\ N^n_{\tau^{b,n}_k}])$, and to each $\tau^{d,n}_k$  we associate a uniform random vector $V_k^n = (V_k^{n}(1), V_k^{n}(2))\sim \text{Unif}([ 1,N^n_{\tau^{d,n}_k}]^2)$. It is clear that both processes are driven by the same marks $(\tau_k^{b,n},U_k^n)_{k\geq 1}$ and $(\tau_k^{d,n},V_k^n)_{k\geq 1}$. Now we analyze separately the behavior of the gaps during intervals in which particles move independently as Brownian motions, and during times when splitting or competition events occur.  On the interval $[0,\tau^{b,n}_1\wedge\tau^{d,n}_1)$, the particle dynamics resemble those of a system of competing Brownian particles, so Corollary 3.11 in \cite{Sarantsev2019} ensures that
	\[X^{1,n}_{(i+1)}(\tau^{b,n}_1\wedge\tau^{d,n}_1 -)-X^{1,n}_{(i)}(\tau^{b,n}_1\wedge\tau^{d,n}_1-)\leq X^{2,n}_{(i+1)}(\tau^{b,n}_1\wedge\tau^{d,n}_1 -)-X^{2,n}_{(i)}(\tau^{b,n}_1\wedge\tau^{d,n}_1-) \]
	for all ranks $i\in[1,n-1]$. At the time of the first event we have two possibilities. If $\tau^{b,n}_1\wedge\tau^{d,n}_1 = \tau_1^{b,n}$, then we update both position vectors and obtain $\Pi_{U_1^n}^+(X^{.,n})$.  At this time, the two particles located at position $X^{.,n}_{(U_1^n)}(\tau_1^{b,n})$ have gap zero (in both systems). Moreover, the newly created gaps between this new particle and its neighbors are inherited from the gaps associated with the particle ranked $U_1^n$, while the gaps between all other ranked particles remain unchanged. Hence, the domination property is preserved. 
	
	On the other hand, if $\sigma^n_1\wedge\tau^n_1 = \tau_1^n$ we update both position vectors and obtain $\Pi_{V_1^n(1)V_1^n(2)}^-(X^{.,n})$. Without loss of generality, assume that the particle suppressed from the system is the one with rank $V_1^n(1)$. If $V_1^n(1)=1$, i.e. the leftmost particle is removed, the gaps remain unchanged. If $V_1^n(1)>1$, then the new gap is formed as the sum of the gaps between the particle ranked  $V_1^n(1)$ and its two neighbors. This ensures that the dominance property is preserved. We can proceed inductively, so that the dominance can be established for any time in the future.
\end{proof}

\begin{proof}[Proof of item ii) of \ref{thm: speed} .]
	\emph{Lower bound.} 
	In Subsection~\ref{subsec: seen-min}, we proved that the system observed from its minimum, $(\mathbf{Z}(t))_{t\geq 0}$ has a unique stationary distribution, denoted by $\boldsymbol{\nu}^c$. We define $c_K :=c/K$ for each $K\geq 1$ and set $\boldsymbol{\nu}^{c_K} = \boldsymbol{\nu}^K$. Denote by $\eta^K_{\nu}$ the initial distribution of the process $\mathbf{X}^K$ in which the particle at the leftmost position is at the origin and the others are arranged according to $\boldsymbol{\nu}^K$. 
	We start by rewriting $\widetilde{F}^K(t,x)$ as 
	\[
	\widetilde{F}^K(t,x) = \frac{1}{N_t^K}\sum_{j=1}^{N^K_t}  \mathbb{1}_{\{X^K_{(j)}(t)\leq x \}}= \mathbb{1}_{\{X^K_{(N_t^K)}(t)\leq x\}} + \sum_{j=1}^{N^K_t-1} \frac{j}{N_t^K} \mathbb{1}_{\{X^K_{(j)}(t)\leq x \leq X^K_{(j+1)}(t)\}},
	\]
	so we can substitute this representation of $\widetilde{F}^K(t,x)$ in \eqref{eq: dervempmean} and obtain the following
	expression in terms of the spacings between each consecutive order statistics of the process
	\begin{equation}
		\frac{\partial}{\partial t}	\mathbb{E}_{\eta^K_{\nu}}\left(\overline{\mathbf{X}}^K(t)\right) =  \mathbb{E}_{\eta^K_{\nu}}\left[\sum_{j=1}^{N^K_t-1} \frac{c_Kj(N_t^K-j)}{N_t^K-1} \left(X^K_{(j+1)}(t) - X^K_{(j)}(t) \right) \right] 
	\end{equation}

	For every $n\in\mathbb{N}$,  denote by $\eta^{n,K}_{\nu}$ the initial distribution $\eta^{K}_{\nu}$ of the process $\mathbf{X}^K$ conditioned to have $n$ particles at the beginning. We can then write  
	\[\mathbb{E}_{\eta^{K}_{\nu}}(\cdot)=\sum_{n\geq 1} \mathbb{E}_{\eta^{n,K}_{\nu}}(\cdot)\pi(n).\]
	Let $\eta_{0}^{K}$ denote the configuration when the system is started from the stationary population size, and let $\eta_{0}^{n,K}$ denote the initial configuration in which $n \in\mathbb{N}$ particles are located at the origin. 
	For each 
	$n$ the gaps between particles in the system started from the distribution $\eta^{n,K}_{\nu}$
	remain, at all times, larger than those in the system started from 
	$\eta_{0}^{n,K}$. This ensures that, by Lemma \ref{gaps}
	\begin{equation*}
		\mathbb{E}_{\eta^{n,K}_{\nu}}\left[\sum_{j=1}^{N^K_t-1} \frac{c_Kj(N_t^K-j)}{N_t^K-1} \left(X^K_{(j+1)}(t) - X^K_{(j)}(t) \right) \right] \geq \mathbb{E}_{\eta^{n,K}_{0}}\left[\sum_{j=1}^{N^K_t-1} \frac{c_Kj(N_t^K-j)}{N_t^K-1} \left(X^K_{(j+1)}(t) - X^K_{(j)}(t) \right) \right]
	\end{equation*}
	and therefore, we have the following domination
	\begin{equation}\label{eq: dom}
		\frac{\partial}{\partial t}	\mathbb{E}_{\eta_{\nu}^{K}}\left(\overline{\mathbf{X}}^K(t)\right) \geq 	\frac{\partial}{\partial t}	\mathbb{E}_{\eta_{0}^{K}}\left(\overline{\mathbf{X}}^K(t)\right) = \int_{-\infty}^\infty \mathbb{E}_{\eta_{0}^{K}}\left[\frac{c_K(N^K_t)^2}{N^K_t-1}\left(\tilde{F}^K(t,x) 
		- \tilde{F}^K(t,x)^2\right)\right]dx.
	\end{equation}		
	
	If we start the system from the initial distribution $\eta_{0}^{K}$, we know that for each $K\geq 1$, $\tilde{F}^K(0,x) = \mathbb{1}_{\{x\leq 0\}}$. Since this initial (random) cumulative function does not depend on $K$, taking the limit $K\to\infty$ yields $\widetilde{F}^K(0,x)\to \widetilde{F}(0,x)= \mathbb{1}_{\{x\leq 0\}}$ in probability.
	Moreover, by Corollary \ref{Col: cdf} and Remark \ref{rem: equalmeas}  $\widetilde{F}^K(t,x)$ converges in probability to $F(t,x)$ with initial condition $F(0,x)= \mathbb{1}_{\{x\leq 0\}}$, which, in this particular case, is the classical solution to the FKPP equation. Therefore, it is known that there exists a nontrivial limiting shape $w_{\sqrt{2}}$ and a centering term $m(t)$ such that 
	\begin{equation}\label{conv}
		F(t,m(t)+z)\to w_{\sqrt{2}}(z)\quad t\to\infty.
	\end{equation}
	uniformly in $z$. If we substitute $w_{\sqrt{2}}(z)$ into equation \eqref{eq: fkpp} for the classical case, with $H(u)=u(u-1)$, we obtain the following ODE:
	\begin{equation}
		\begin{cases}
			w_{\sqrt{2}}'' + \sqrt{2}\, w_{\sqrt{2}}' + w_{\sqrt{2}}(w_{\sqrt{2}} - 1) = 0, \\
			w_{\sqrt{2}}(-\infty) = 0, \\
			w_{\sqrt{2}}(+\infty) = 1.
		\end{cases}
	\end{equation}
	Integrating with respect to 
	$z$, we obtain the identity 
	\begin{equation}\label{eq: sqrt2}
		\int_{\mathbb{R}}w_{\sqrt{2}}(z)(1-w_{\sqrt{2}}(z))dz=\sqrt{2}.
	\end{equation}
	
	On the other hand, we can take the expectation of the empirical mean $\overline{\mathbf{X}}^K$ w.r.t. $\eta^K_{\nu}$ and add and subtract  the expectation of the minimum w.r.t. $\eta^K_{\nu}$, to obtain
	\begin{equation}\label{eq: meanempmean}
		\mathbb{E}_{\eta^K_{\nu}}	\left(\overline{\mathbf{X}}^K(t)\right) = 
		\mathbb{E}_{\eta^K_{\nu}} \left(\frac{1}{N^K_t} \sum_{j= 1}^{N^K_t} \left(X^K_{(i)}(t)  - X^K_{(1)}(t) \right)\right) + \mathbb{E}_{\eta^K_{\nu}}\left(X^K_{(1)}(t)\right).
	\end{equation}
	The first term in \eqref{eq: meanempmean} is the empirical mean of the process $(\mathbf{Z}(t))_{t\geq 0}$ starting from the stationary distribution $\boldsymbol{\nu}^K$, so its expectation will not depend on time. On the other hand, the process $(X_{(1)}(t))_{t\geq 0}$ has stationary increments, and therefore its mean with respect to $\eta^K_{\nu}$ is linear with slope equal to $v_K$, by item i) of Theorem \ref{thm: speed}. So, if we take the derivative with respect to $t$ on both sides of \eqref{eq: meanempmean}, we obtain that 
	\begin{equation}\label{eq:slope}
		\frac{\partial}{\partial_t}\mathbb{E}_{\eta^K_{\nu}}	\left(\overline{\mathbf{X}}^K(t)\right) = v_K
	\end{equation}  
	Thanks to \eqref{eq: dom}, taking $0< b<\infty$ we have that 
	\begin{equation}\label{eq: lim1}
		v_K = \frac{\partial}{\partial t}	\mathbb{E}_{\eta_{\nu}^K}\Big(\overline{\boldsymbol{X}}^K(t)\Big) \geq \int_{-b+m(t)}^{b+m(t)} \mathbb{E}_{\eta_0^K}\Bigg[\frac{c_K(N^K_t)^2}{N^K_t-1}\Big(\widetilde{F}^K(t,x) 
		- \widetilde{F}^K(t,x)^2\Big)\Bigg]dx.
	\end{equation}
	Since $\widetilde{F}^K$ is uniformly bounded by one, the bounded convergence theorem implies that $\widetilde{F}^K$ also converges to $F$ in $L^1$. Consequently, $\widetilde{F}^K(t,x) - \widetilde{F}^K(t,x)^2$ converges to $F(t,x) 
	- F(t,x)^2$ in $L^1$. Moreover, we know that $\frac{c_K{(N^K_t)}^2}{N^K_t-1}\to 1$ in $L^1$ and therefore 
	\[\lim_{K\to\infty}\mathbb{E}_{\eta_0^K}\left[\frac{c_K(N^K_t)^2}{N^K_t-1}\left(\tilde{F}^K(t,x) 
	- \tilde{F}^K(t,x)^2\right)\right]= F(t,x)(1-F(t,x)),\]
	so taking limits as $K\to \infty$ on both side of \eqref{eq: lim1}, and performing the change of variables $x=z+m(t)$, we obtain
	\[
	v_0 \geq \lim_{t\to\infty}\int_{-b}^{b} (F(t,z+m(t))(F(t,z+m(t))-1))dz = \int_{-b}^{b} w_{\sqrt{2}}(z)(1-w_{\sqrt{2}}(z))dz,
	\]
	where the last equality is given by \eqref{conv}. We can conclude the proof of the lower bound by taking the supremum over $b$ thanks to \eqref{eq: sqrt2}.
	
	\noindent
	\emph{Upper Bound}. Consider the coupling construction between the Log-BBM($1,c$) (with initial configuration given by $\xi_{N_0^c}$) and a system of $N_0^c$ dyadic branching Brownian motions defined in Subsection \ref{subsec: coupling}. If we denote by $M_t^i$ the position of the rightmost particle at time $t$ in the $i$-th branching Brownian motion, we have
	\[\max_{u\in \mathcal{N}^c_t}X^c_u(t)\leq \max_{u\in \mathcal{M}_t}X_u(t)=\max\{M^1_t,\cdots,M^{N_0^c}_t\}\quad \text{a.s. for}\quad t\geq 0.\]
	Moreover, by \cite{Champneys1995} it is known that, for any $i\in[1,N_0^c]$,  
	\[\lim_{t\to\infty}t^{-1}M^i_t=\sqrt{2}
	\qquad \text{a.s.}\]
	Therefore, we obtain the desired upper bound.
\end{proof}

\section*{Acknowledgements}
This project was supported by the UNAM-PAPIIT grant IN109924 {\it ``Comportamiento de poblaciones con interacción y competencia''}.

\appendix 
\section{Proof of Lemma~\ref{lem: martngl} }\label{sec: AppdxA}
\begin{lemma}\label{lem: order}
	Let $(|\mathcal{M}_t|)_{t\geq 0}$ denote the population size at time $t$ of $N^K$ independent branching Brownian motions. Assume that, for $K \in \mathbb{N}$, $N^K$ is a random variable with the distribution of a Poisson random variable with parameter $K/c$, conditioned to be positive. Then, for every fixed $t> 0$ and $m\in\mathbb{N}$
	\[\mathbb{E}\big(|\mathcal{M}_t|^m\big) = O_t(K^m)\text{ as }K\to \infty.\]
\end{lemma}
\begin{proof}

	For each $t\geq 0$, let $M_{|\mathcal{M}_t|}(\lambda)$ denote the moment generating function of $|\mathcal{M}_t|\sim \text{NB}(N^K,e^{-t})$, conditional on the variable $N^K$.
	We know that, for $\lambda<-\log(1-e^{-t})$, $M_{|\mathcal{M}_t|}(\lambda) = \mathbb{E}\left[f(\lambda)^{N^K}\right]$ with 
	\[f(t,\lambda)=\frac{e^{-t}}{1-(1-e^{-t})e^{\lambda}}\ .\] The function $f(t,\lambda)$ is infinitely differentiable in $\lambda$, therefore $M_{|\mathcal{M}_t|}(\lambda)$ is also infinitely differentiable in $\lambda$, and using Faa di Bruno's formula we obtain 
	\[\frac{\partial^m M_{|\mathcal{M}_t|}(\lambda)}{\partial \lambda^m} = \mathbb{E}\Bigg(\sum_{j=1}^{m\wedge N^K} \frac{N^K!}{(N^K-j)!}f(t,\lambda)^{N^K-j}B_{m,j}(\partial_{\lambda}f(t,\lambda), \cdots, \partial_{\lambda}^{(m-j+1)}f(t,\lambda))\Bigg),\]
	where $B_{m,j}$ are the Bell polynomials. Given that $f(0)=1$ and $\partial_{\lambda}^{(j)}f(t,0)$ is finite for every $j$, then for $\lambda=0$ we have
	\begin{equation}\label{eq: moment}
		\begin{split}
			\mathbb{E}\Big(|\mathcal{M}_t|^m\Big) &= \mathbb{E}\Bigg(\sum_{j=1}^{m\wedge N^K} \frac{N^K!}{(N^K-j)!}B_{m,j}(\partial_{\lambda}f(t,0), \cdots, \partial_{\lambda}^{(m-j+1)}f(t,0))\Bigg)\\
			& \leq B_{m,j}(\partial_{\lambda}f(t,0), \cdots, \partial_{\lambda}^{(m-j+1)}f(t,0))m\mathbb{E}\Bigg(\frac{N^K!}{(N^K-m\wedge N^K)!}\Bigg)
		\end{split}.
	\end{equation}
	In particular, on the event $\{m\leq N^K\}$, the expectation in \eqref{eq: moment} equals the $m$-th factorial moment of $N^K$ and is therefore given by  $(1-e^{-K/c})^{-1}(K/c)^m$, which is of order $K^m$. As the probability of the event $\{m> N^K\}$ converges to zero as $K \to \infty$, we can conclude the desired result.
\end{proof}

\begin{proof}[Proof Lemma~\ref{lem: martngl}]

	For all $\phi(t,x) \in C^{1,2}_b(\mathbb{R}_+\times\mathbb{R})$, we have
	\begin{equation}\label{eq: incrmt}
		\left< \phi, \mu_t\right> - \left< \phi, \mu_0\right> = \frac{1}{m_K} \left[\sum_{u \in \mathcal{N}^K_t} \phi(t,X^K_u(t))  -  \sum_{u \in \mathcal{N}^K_0} \phi(0,X^K_u(0))\right].
	\end{equation}
	We can divide the particles appearing in \eqref{eq: incrmt} into three groups: those that remain alive throughout the interval $[0,t]$ (i.e. $u\in\mathcal{N}^K_0$ with $\tau^d_u>t$ ); those that were initially present in the system but died before time $t$ (i.e. $u\in\mathcal{N}^K_0$ with $ 0<\tau^d_u \leq t$); and those who were born during $(0,t]$ and survive until time $t$ (i.e. $u\in \mathbb{T}^K$ with $0<\tau^b_u \leq t<\tau^d_u$). Consequently, \eqref{eq: incrmt} can be rewritten as follows:
	\begin{equation}\label{eq: incrmt1} 
		\begin{split}
			m_K\left(\left< \phi, \mu_t\right> - \left< \phi, \mu_0\right> \right) =& \sum_{u\in\mathcal{N}^K_0: \tau^d_u>t} \phi(t,X^K_u(t))  -   \phi(0,X_u(0))\\  &+ \sum_{u\in\mathbb{T}^K: 0<\tau^b_u \leq t<\tau^d_u} \phi(t,X^K_u(t)) 
			- \sum_{u\in\mathcal{N}^K_0: 0<\tau^d_u \leq t} \phi(0,X^K_u(0)).
		\end{split}
	\end{equation}
	For completeness, in the second term on the right-hand side of \eqref{eq: incrmt} we may add and subtract the quantity $ \phi(\tau^{b}_u,X^K_u(\tau^{b}_u))$ for all $u\in\mathbb{T}^K$ such that $0<\tau^b_u \leq t$. Such particles may survive up to time $t$ (and possibly beyond) or die before time $t$. With this observation, the second term on the right-hand side of \eqref{eq: incrmt} can be rewritten as
	\begin{equation}\label{eq: incrmt2 }
		\begin{split}
			\sum_{u\in\mathbb{T}^K: 0<\tau^b_u \leq t<\tau^d_u} \phi(t,X^K_u(t)) 
			& = \sum_{u\in\mathbb{T}^K:0<\tau^b_u \leq t<\tau^d_u} \phi(t,X^K_u(t)) - \phi(\tau^{b}_u,X^K_u(\tau^{b}_u))
			\\
			& - \sum_{u\in\mathbb{T}^K:0<\tau^b_u < \tau^d_u <t} \phi(\tau^{b}_u,X^K_u(\tau^{b}_u)) + \sum_{u\in\mathbb{T}^K:0<\tau^b_u \leq t} \phi(\tau^{b}_u,X^K_u(\tau^{b}_u))
		\end{split}
	\end{equation}
	Proceeding in the same way, in the third term of \eqref{eq: incrmt} we add and subtract $\phi(\tau^{d}_u,X^K_u(\tau^{d}_u-))$ for all $u\in\mathbb{T}^K$ such that $0<\tau^d_u \leq t$. We then divide these particles into two groups: the initial particles that died before time 
	$t$, and those that were born during $(0,t)$ and died before time $t$. This gives us the following representation of the third term:
	\begin{equation}\label{eq: incrmt3 }
		\begin{split}
			\sum_{u\in\mathcal{N}^K_0: 0<\tau^d_u \leq t} \phi(0,X^K_u(0)) 
			&= \sum_{u\in\mathcal{N}^K_0: 0<\tau^d_u \leq t} \phi(0,X^K_u(0))-\phi(\tau^{d}_u,X^K_u(\tau^{d}_u-))  \\
			&  + \sum_{u\in \mathbb{T}^K:0<\tau^d_u \leq t} \phi(\tau^{d}_u,X^K_u(\tau^{d}_u-))- \sum_{u\in \mathbb{T}^K:0<\tau^b_u < \tau^d_u \leq t} \phi(\tau^{d}_u,X^K_u(\tau^{d}_u-)).
		\end{split}
	\end{equation}
	Putting together \eqref{eq: incrmt1}, \eqref{eq: incrmt2 }, \eqref{eq: incrmt3 }, we obtain
	\begin{equation}\label{eq: incrmt4}
		\begin{split}
			m_K\big(\big< \phi, \mu_t\big> - \big< \phi, \mu_0\big> \big) =& \sum_{u\in\mathbb{T}^K: 0 \leq \tau_u^b\leq t} \phi(t\wedge \tau_u^d,X^K_u(t\wedge \tau_u^d-))-\phi(\tau^b_u,X^K_u(\tau_u^b))  \\  &+ \sum_{u\in\mathbb{T}^K:0<\tau^b_u \leq t} \phi(\tau^{b}_u,X^K_u(\tau^{b}_u))
			- \sum_{u\in \mathbb{T}^K:0<\tau^d_u \leq t} \phi(\tau^{d}_u,X^K_u(\tau^{d}_u-)).
		\end{split}
	\end{equation}
	\noindent
	Since every particle $u\in\mathbb{T}^K$ so that $0 \leq \tau_u^b\leq t$ evolves as a Brownian motion in the interval 
	$\tau^b_u\leq s \leq \tau_u^d\wedge t $, we may apply Ito’s formula to each term in the sum appearing in the first term on the right-hand side of \eqref{eq: incrmt4}. In doing so, we obtain the following.
	\begin{equation}\label{eq: incrmt5}
		\begin{split}
			\sum_{\substack{
					u \in \mathbb{T}^K \\
					0 \le \tau_u^b \le t}} \phi(t\wedge \tau_u^d,X^K_u(t\wedge \tau_u^d-))&-\phi(\tau^b_u,X^K_u(\tau_u^b))
			= \sum_{\substack{
					u \in \mathbb{T}^K \\
					0 \le \tau_u^b \le t
			}}\int_{\tau_u^b}^{t\wedge \tau^{d}_u} \partial_x\phi(s, X^K_u(s)) dX^K_u(s)\\
			&+  \sum_{\substack{
					u \in \mathbb{T}^K \\
					0 \le \tau_u^b \le t
			}} \int_{\tau_u^b}^{t\wedge \tau^{d}_u}\Big[\partial_t\phi(s, X^K_u(s)) +\frac{1}{2}  \partial_{xx}\phi(s, X^K_u(s))\Big] ds
		\end{split}
	\end{equation}
	Substituting \eqref{eq: incrmt5} into \eqref{eq: incrmt4}, and adding and subtracting $D_t^K$ and $B_t^K$ we obtain \eqref{eq: stochdiff}. It remains to verify that the processes defined in \eqref{eq: martings} are indeed square-integrable martingales. We start with the process $(M^K_t)_{t\geq 0}$ 
	\begin{equation}\label{hl.19}
		M^K_t = \frac{1}{m_K}\sum_{u\in\mathbb{T}^K: 0 \leq \tau_u^b\leq t}\int_{0}^{t} \mathbb{1}_{\{\tau^{b}_u\leq s< \tau^{d}_u\}}\partial_x\phi(s, X^K_u(s)) dX^K_u(s) 
	\end{equation} 
	Given that $\phi$ is bounded and on $\{\tau^{b}_u\leq s< \tau^{d}_u\}$ the particle $u$ behaves as a Brownian motion, the integrals in \eqref{hl.19} are Ito's integrals and therefore martingales. Furthermore, using Itô’s isometry and the fact that, by the coupling defined in Section \ref{subsec: coupling}, the number $\#\big\{u\in\mathbb{T}^K:0 \leq \tau_u^b\leq t\big\}$ of births on $[0,t]$ in the Log-BBM($1,c_K$)  can be bounded above by the number of particles in a standard BBM with a birth rate one and an initial mass $m_K$, we obtain, for each $t\geq 0$
	\[
	\begin{split}
		\lim_{K\rightarrow\infty}\mathbb{E}\Big[\big(M_t^K\big)^2\Big] &\leq m_K^{-2}t\big\|\big(\partial_x\phi\big)^2\big\|_{\infty}\mathbb{E}\big[\#\big\{u\in\mathbb{T}^K:0 \leq \tau_u^b\leq t\big\}\big] \\
		&\leq \lim_{K\rightarrow\infty} m_K^{-1}t\big\|\big(\partial_x\phi\big)^2\big\|_{\infty} e^t  =0.
	\end{split}
	\]
	Thus, we can conclude that $M_t^K$ is a square-integrable martingale that vanishes as $K\to \infty$. Also, by the Doob-Meyer decomposition and Ito's isometry we obtain that
	\begin{equation}\label{hl.19}
		[M^K]_t = \frac{1}{m_K^2}\sum_{u\in\mathbb{T}^K: 0 \leq \tau_u^b\leq t}\int_{0}^{t} \mathbb{1}_{\{\tau^{b}_u\leq s< \tau^{d}_u\}}\big(\partial_x\phi(s, X^K_u(s))\big)^2 ds. 
	\end{equation} 
	We now turn to the process $(D^K_t)_{t\geq 0}$, and fix $K\geq 1$. 
	Let $\{\mathcal{F}^K_t\}$ be the natural filtration associated to $\boldsymbol{X}^K$.
	For a fix $s\leq t$, define 
	\begin{equation*}\label{eq: diffD}
		f(t) := \mathbb{E}(D_t^{K} - D_s^{K}| \mathcal{F}^K_s).
	\end{equation*} 
	We can compute the derivative of  $f(t)$ directly from its definition as follows
	\begin{equation}\label{eq: derD}
		\begin{split}
			f'(t)=\lim_{h\rightarrow 0}\frac{1}{h} \mathbb{E}\big(D_{t+h}^{K} - D_t^{K}\big| \mathcal{F}^K_s\big)
			& = \lim_{h\rightarrow 0}\frac{1}{h}\mathbb{E}\Bigg(\frac{1}{m_K}\sum_{\substack{
					u \in \mathbb{T}^K \\
					t \le \tau_u^d \le t+h
			}} \phi(\tau^{d}_u,X^K_u(\tau^{d}_u-))\\ & - \int_t^{t+h}   \frac{2c_K}{m_K}  \sum_{\stackrel{u,v \in \mathcal{N}^K_w}{v\neq u}}  \phi(w,X^K_v(w))\mathbb{1}_{X^K_v(w)\leq X^K_u(w)}dw\Bigg|\mathcal{F}^K_s\Bigg) 
		\end{split}
	\end{equation}
	Consider the times $\tau^d_{1} = \inf\{s>t: \exists u \in \mathbb{T}^K \ \text{such that} \ \tau_u^d=s\}$ and $\tau^b_{1} = \inf\{s>t: \exists u \in \mathbb{T}^K \ \text{such that} \ \tau_u^b=s\}$, which respectively denote the first death time and the first birth time after $t$. Also, let $u_1\in\mathbb{T}^K$ be such that $\tau_{u_1}^d = \tau^d_{1}$, and denote $X_{u_1}^K(\tau_1^d-)$ its position just before death. For the first term in the r.h.s. of \eqref{eq: derD}, since over a time interval of length $h$ (for $h$ sufficiently small), the probability of observing two or more events is of order $h^2$, we obtain 
	\begin{equation}\label{eq: firsttime}
		\mathbb{E}\Bigg(\frac{1}{m_K}\sum_{\substack{
				u \in \mathbb{T}^K \\
				t \le \tau_u^d \le t+h
		}}  \phi(\tau^{d}_u,X^K_u(\tau^{d}_u-))\Bigg|\mathcal{F}^K_t\Bigg)= \mathbb{E}\Bigg(\frac{1}{m_K}\phi(\tau^{d}_1,X^K_{u_1}(\tau^{d}_1-))\mathbb{1}_{\tau_1^b>t+h}\Bigg|\mathcal{F}^K_t\Bigg) +O_K(h^2)
	\end{equation}
	Let $\overline{c_K} := c_K N^K_t(N^K_t-1)$ and $\kappa(N^K_t) := N_t^K + \overline{c_K}$. 
	Conditional on the event $\{\tau_1^b>t+h\}$, the time $\tau_1^d -t$ has density  $f_{\tau^{d}_1 -t}(h) =  \overline{c_K}e^{-\kappa(N^K_t)h}$. Moreover, at time $\tau_1^d$, a pair of particles $(u,v)$ is uniformly selected among all ordered pairs in $\mathcal{N}_t^K$, that is, with probability $c_K/\overline{c_K}$, and we  set $u_1 = k_{u,v}(\tau_1^d)$. Therefore, the first term in the rhs of \eqref{eq: firsttime} can be written as
	\begin{equation}\label{eq: jumpart}
		\begin{split}
			&\mathbb{E}\Bigg(\frac{1}{m_K} \sum_{\stackrel{v,u \in\mathcal{N}^K_t}{u\neq v}}\int_t^{t+h}  \phi(w,X_{k_{u,v}(w)}(w))\mathbb{P}(u,v\ \text{compete})f_{\tau_1^d -t}(w-t)dw \Bigg| \mathcal{F}^K_t \Bigg)\\
			&= \mathbb{E}\Bigg(\frac{1}{m_K} \sum_{\stackrel{v,u \in\mathcal{N}^K_t}{u\neq v}}\int_t^{t+h}  c_K\phi(w,X_{k_{u,v}(w)}(w)) dw\Bigg| \mathcal{F}^K_t\Bigg) +O_K(h^2)\\
			& = \mathbb{E}\Bigg(\frac{2c_K}{m_K} \int_t^{t+h}\Bigg( \sum_{u \in\mathcal{N}^K_t} \phi(w,X^K_u(w))\Bigg(\sum_{v \in\mathcal{N}^K_t}\mathbb{1}_{X^K_u(w)\leq X^K_v(w)} -1\Bigg)\Bigg)\Bigg| \mathcal{F}^K_t\Bigg) +O_K(h^2),
		\end{split}
	\end{equation}
	Here, the first equality holds since $\mathbb{P}(u,v\ \text{compete})f_{\tau_1^d -t}(w-t) = c_K + O_K(h)$ and $\phi$ is bounded. On the other hand, for the second term in the rhs of \eqref{eq: derD}, an infinitesimal analysis on the interval $[t,t+h]$ gives us that it equals to 
	\begin{equation*}
		\mathbb{E}\Bigg(\frac{2c_K}{m_K} \int_t^{t+h}\Bigg( \sum_{u \in\mathcal{N}^K_t} \phi(w,X^K_u(w))\Bigg(\sum_{v \in\mathcal{N}^K_t}\mathbb{1}_{X^K_u(w)\leq X^K_v(w)} -1\Bigg)\Bigg)\Bigg| \mathcal{F}^K_t\Bigg) +O_K(h^2)
	\end{equation*}
	Therefore, we can conclude from equations \eqref{eq: derD}-\eqref{eq: jumpart} that $f'(t) =0$ for all $t\geq 0$. Since $f(s)=0$, it follows that $f(t)=0$ for all $t\geq s$.    
	
	It remains to verify that the process $(D^K_t)$ is square–integrable. 
	In fact, by definition
	\begin{equation}\label{eq: quadvarD}
		\begin{split}
			\mathbb{E}\Big(\big|m_KD^{K}_t\big|^2\Big) &\leq \mathbb{E}\Bigg[\Bigg(\sum_{\substack{
					u \in \mathbb{T}^K \\
					0 \le \tau_u^d \le t
			}}  \phi(\tau^{d}_u,X^K_u(\tau^{d}_u-))\Bigg)^2\Bigg]\\
			&+4\mathbb{E}\Bigg[\Bigg|\sum_{\substack{
					u \in \mathbb{T}^K \\
					0 \le \tau_u^d \le t }}  \phi(\tau^{d}_u,X^K_u(\tau^{d}_u-)) \int_0^t  c_K  \sum_{\stackrel{u,v \in \mathcal{N}^K_w}{u\neq v}} \phi(w,X^K_u(w)) \mathbb{1}_{X^K_u(w)\leq X^K_v(w)}dw\Bigg|\Bigg]\\
			& + \mathbb{E}\Bigg[\Bigg(\int_0^t  2c_K  \sum_{\stackrel{u,v \in \mathcal{N}^K_w}{u\neq v}} \phi(w,X^K_u(w)) \mathbb{1}_{X^K_u(w)\leq X^K_v(w)}dw\Bigg)^2\Bigg]
		\end{split} \quad.
	\end{equation}
	By the coupling described in Subsection \ref{subsec: coupling}, we have $\big|\mathcal{M}_t\big|\geq \#\big\{u\in\mathbb{T}^K:0< \tau_u^d \leq t\big\}$  for $t\geq 0$.  Therefore, for the first term in \eqref{eq: quadvarD} we obtain 
	\[\begin{split}
		\mathbb{E}\Bigg[\Bigg(\sum_{\substack{
				u \in \mathbb{T}^K \\
				0 \le \tau_u^d \le t
		}} \phi(\tau^{d}_u,X^K_u(\tau^{d}_u-))\Bigg)^2\Bigg] \leq \|\phi\|^2_{\infty}\mathbb{E}\Big(\big|u\in\mathbb{T}^K:0< \tau_u^d \leq t\big|^2\Big)\leq\|\phi\|^2_{\infty}\mathbb{E}\big({\big|\mathcal{M}_t\big|}^2\big).
	\end{split}\]
	For the second term, we have a similar domination: $\big|\mathcal{M}_t\big| \geq |\mathcal{N}^K_w|$ for $t\geq w$, and hence
	\[
	\begin{split}
		&4\mathbb{E}\Bigg[\Bigg|\sum_{\substack{
				u \in \mathbb{T}^K \\
				0 \le \tau_u^d \le t
		}}  \phi(\tau^{d}_u,X^K_u(\tau^{d}_u-)) \int_0^t  c_K  \sum_{\stackrel{u,v \in \mathcal{N}^K_w}{u\neq v}} \phi(w,X^K_u(w)) \mathbb{1}_{X^K_u(w)\leq X^K_v(w)}dw\Bigg|\Bigg]\\
		&\leq 4c_K\|\phi\|^2_{\infty}\mathbb{E}\Big[\big|u\in\mathbb{T}^K:0< \tau_u^d \leq t\big|\int_0^t|\mathcal{N}^K_w|(|\mathcal{N}^K_w|-1)dw\Big]\leq 4tc_K\|\phi\|^2_{\infty}\mathbb{E}\Big({\left|\mathcal{M}_t\right|}^3\Big)
	\end{split} \quad;
	\]
	and similarly for the third term
	\[
	\mathbb{E}\Bigg[\Bigg(\int_0^t  2c_K  \sum_{\stackrel{u,v \in \mathcal{N}^K_w}{u\neq v}} \phi(w,X^K_u(w)) \mathbb{1}_{X^K_u(w)\leq X^K_v(w)}dw\Bigg)^2\Bigg]
	\leq 4t^2c^2_K\|\phi\|^2_{\infty}\mathbb{E}\big({\big|\mathcal{M}_t\big|}^4\big) \,.
	\]
	By Lemma \ref{lem: order}, for $n \in\{ 2,3,4\}$, the moment $\mathbb{E}\big({\big|\mathcal{M}_t\big|}^n\big)$ is of the order $O(K^n)$. Consequently, for $K$ sufficiently large
	\[\mathbb{E}\big[(D^{K}_t)^2\big]\leq m_K^{-2}\|\phi\|^2_{\infty}\Big(\mathbb{E}\big[{\left|\mathcal{M}_t\right|}^2\big]+ 4tc_K\mathbb{E}\big[{\big|\mathcal{M}_t\big|}^3\big] + 4t^2c^2_K\mathbb{E}\big[{\big|\mathcal{M}_t\big|}^4\big]\Big)\leq \|\phi\|^2_{\infty} C,\]
	where $C:=C(t,c)$. So, $\sup_{t\leq T}\mathbb{E}\big[(D^{K}_t)^2\big]\leq\|\phi\|^2_{\infty}C(T,c)$.
	We have shown that $(D_t^K)$ is a square-integrable martingale with $D_0^K = 0$.  Now, since the integral part of $D^K_t$ has finite variation  its contribution to the quadratic variation vanishes. Therefore, the quadratic variation $[D^K]_t$ comes entirely from the jump part of the process, and in particular,
	\[ [D^K]_t = \frac{1}{m^2_K}\sum_{\substack{
			u \in \mathbb{T}^K \\
			0 \le \tau_u^d \le t
	}}  \phi^2(\tau^{d}_u,X^K_u(\tau^{d}_u-)).\]
	Therefore, using a similar approach as for proving that $D^K_t$ is a martingale, we can prove that 
	\[[D^K]_t - \int_0^t  \frac{2c_K}{m^2_K}  \sum_{u \in \mathcal{N}^K_w} \phi^2(w,X^K_u(w))\Bigg(\sum_{v \in \mathcal{N}^K_w}\mathbb{1}_{X^K_u(w)\leq X^K_v(w)}-1\Bigg)dw\]
	is a martingale, and because the integral in the expression above is predictable, the Doob–Meyer decomposition ensures that this integral is the predictable quadratic variation, that is
	\[\left<D^K\right>_t = \int_0^t  \frac{2c_K}{m^2_K}  \sum_{u \in \mathcal{N}^K_w} \phi^2(w,X^K_u(w))\Bigg(\sum_{v \in \mathcal{N}^K_w}\mathbb{1}_{X^K_u(w)\leq X^K_v(w)}-1\Bigg)dw.\]
	Hence, using the fact that $\mathbb{E}[(D_t^K)^2] = \mathbb{E}\big(\big<D^K
	\big>_t\big)$ it follows that
	\[
	\begin{split}
		\mathbb{E}[(D_t^K)^2] &= \mathbb{E}\Bigg[\int_0^t  \frac{2c_K}{m^2_K}  \sum_{u \in \mathcal{N}^K_w} \phi^2(w,X^K_u(w))\Bigg(\sum_{v \in \mathcal{N}^K_w}\mathbb{1}_{X^K_u(w)\leq X^K_v(w)}-1\Bigg)dw\Bigg]\\
		& \leq \frac{2tc_K}{m^2_K}\|\phi\|^2_{\infty}\mathbb{E}\big(\big|\mathcal{M}_t\big|^2\big)\leq c_K\|\phi\|^2_{\infty}\widetilde{C}\longrightarrow 0 \qquad \text{as} \quad K\longrightarrow 0,
	\end{split}
	\]
	for some constant $\widetilde{C}$. Then $D_t^K \to 0$ in $L^2$. 
	
	An analogous argument applies to $B_t^K$. In particular, its predictable quadratic variation is given by
	\[\big<B^K\big>_t = \int_0^t \frac{1}{m_K^2}\sum_{u \in \mathcal{N}^K_s}\phi^2(s,X^K_u(s))ds.\]
\end{proof}

\section{Proof of Lemma \ref{lem: vagueapprox}}\label{appendix: vagueapprox}
\noindent
For item (a) in Lemma \ref{lem: vagueapprox}, a similar property was studied in \cite[Proposition $2.7$]{Meleard2015} for the process $(N_t)_{t\geq 0}$ (without rescaling). However, for each $K\in\mathbb{N}$, a straightforward adaptation suffices for the rescaled process $(K^{-1}N_t^K)_{t\geq 0}$. For items (b) and (c), the proof follows closely the arguments of \cite{Fontbona2013} and \cite{Jourdain2011},
with only minor modifications. We include it here for completeness. 
\begin{itemize}
	\item[a)] For a stopping time $T_n = \inf\{t>0: K^{-1}N_t^K \geq n\}$ there exists a constant $C$ that does not depend on $n$ such that  
	\[\mathbb{E}\Big(\sup_{t\leq T}K^{-p}\big(N_{t\wedge T_n}^K\big)^p\Big)\leq \mathbb{E}\big(K^{-p}\big(N_0^K\big)^p\big) +C\int_0^T \mathbb{E}\Big(\sup_{s\in[0,t\wedge T_n]}K^{-p}\big(N^K_s\big)^p\Big)dt\]
	using Gronwall's Lemma, there exists a constant $\tilde{C}:= \tilde{C}(p,T) $ such that 
	\[\mathbb{E}\Big(\sup_{t\leq T\wedge T_n}K^{-p}\big(N_{t}^K\big)^p\Big)\leq \tilde{C}\mathbb{E}\big(K^{-p}\big(N_0^K\big)^p\big) \]
	As the sequence $(T_n)_n$ tend to infinite a.s., then taking the supremum over $K$ we obtain the desired result.
	\item[b)] The selection of the function $\psi$ ensures that 
	$\psi_0 = 1$, and for $n> 1$,
	$\psi_n = 0$ in $[-n+1,n-1]$ and  $\psi_n = 1$ in $[-n,n]^c$ and also, that both $(\psi_n)_{xx}$ and $(\psi_n)_{x}$ are bounded. Therefore, there exists a constant $C':=C'(K,c)$  such that
	\begin{equation}
		\begin{split}
			\big< \psi_n, \mu^{K}_t\big>   \leq \big< \psi_n, \mu^{K}_0\big> +  C' \int_{0}^{t}   K^{-1}N^K_s+ K^{-2}(N^K_s)^2ds +\int_0^t C'\big<\psi_n,\mu^K_s\big>ds  +  L^{K,\psi_n}_t
		\end{split}
	\end{equation}
	and 
	\begin{equation}
		\begin{split}
			\big< L^{K,\psi_n}\big>_t\leq T C' \Big( \sup_{t\leq T}K^{-1}N^K_s+\sup_{t\leq T}K^{-2}(N^K_s)^2\Big) 
		\end{split}
	\end{equation}
	The constant $C'$ is decreasing in $K$, and $\sup_K C'(K,c)<\infty$, so we can bound it by a constant $C'$ that does not depend on $K$.
	As a consequence of the Buckholder-Davis-Gundy inequality and \cite[Item ($4.b$'), Table $4.1$, p. $162$]{Barlow1986}
	we have that there exists constants $C''$ and $\tilde{C}''$ such that 
	\[\mathbb{E}\Big(\sup_{t\leq T}L^{K,\psi_n}_t\Big)\leq C''\mathbb{E}\Big([L^{K,\psi_n}]^{1/2}_T\Big)\leq \tilde{C}''\mathbb{E}\Big(\big<L^{K,\psi_n}\big>^{1/2}_T\Big)\]
	By part $a)$, there exists a constant $C$ such that
	\begin{equation}
		\begin{split}
			\mathbb{E}\Big(\sup_{t\leq T}\big< \psi_n, \mu^{K}_t\big>\Big)  
			\leq \big< \psi_n, \mu^{K}_0\big>  +C\int_0^T \mathbb{E}\Big(\sup_{s\leq t}\big<\psi_n,\mu^K_s\big>\Big)dt  
			+ CT \mathbb{E}\big(1+K^{-2}(N^K_0)^2\big)
		\end{split}
	\end{equation}
	Consequently, applying Gronwall’s lemma yields the existence of a constant $\tilde{C}$ depending on  $\sup_K \mathbb{E}\big(K^{-2}(N^K_0)^2\big)$ and on $T$ such that 
	\[\sup_K\mathbb{E}\Big(\sup_{t\leq T}\big< \psi_n, \mu^{K}_t\big>\Big)  \leq \tilde{C} \sup_K\mathbb{E}\big(\big< \psi_n, \mu^{K}_0\big>\big) \]
	We know that $\big<\psi_n, \mu^K_0\big>$ converges in law to $\big<\psi_n, \mu_0\big>$ as $K\to \infty$ and also, by part $a)$, it is uniformly integrable, therefore $\mathbb{E}\big(\big< \psi_n, \mu^{K}_0\big>\big)$ converges to $\big<\psi_n, \mu_0\big>$ as $K\to \infty$. Taking the limit as $n\to \infty$, we can conclude.
	\item[c)] For each $n,l\in\mathbb{N}$ define the function $\psi_{n,l}\in C^2_c({\mathbb{R}})$ such that   $\psi_{n,l}:=\psi_n(1-\psi_l)$. We know that $\psi_{n,l}\leq \psi_n\leq 1$, therefore by part $a)$ we can ensure that the sequence $\sup_{t\leq T}\big<\phi_{n,l},\mu_t\big>$ is uniform integrable. Hence, given the continuity of the map $\nu\to \sup_{t\leq T}\big<\psi_{n,l}, \nu_t\big>$ in the space $D([0,T], M^v_F(\mathbb{R}))$ we have
	\[\lim_{K\to \infty} \mathbb{E}\Big(\sup_{t\leq T}\big<\psi_{n,l}, \mu^K_t\big>\Big) = \mathbb{E}\Big(\sup_{t\leq T}\big<\psi_{n,l},\mu_t \big>\Big)\]
	taking limit in $l\to \infty$, it yields
	\[\mathbb{E}\Big(\sup_{t\leq T}\big<\psi_{n}, \mu_t\big>\Big) \leq \lim_{K\to \infty} \mathbb{E}\Big(\sup_{t\leq T}\big<\psi_{n}, \mu^K_t\big>\Big) \]
	Now we can take the limit in $n$ and due to part $b)$ we can conclude that 
	\[\lim_{n\to \infty}\mathbb{E}\Big(\sup_{t\leq T}\big<\psi_{n}, \mu_t \big>\Big) = 0 \]
\end{itemize} 

\section{Proof of Lemma \ref{lem3}}\label{appendix: lem3}
The following is a well‑known result that will be used in the proof of Lemma \ref{lem3}. 
\begin{lemma}[Many-to-One]\label{lem: manytone}
	For any $t\geq 0$, let $F:C([0,t],\mathbb{R})\to\mathbb{R}$ be a measurable function and $(B(s),s\geq 0)$ denote a standard Brownian motion. Then, if $\boldsymbol{X}(t): = \{X_u(t), u\in\mathcal{M}_t\}$ is a BBM starting with $N$ particles located at the origin, it holds that
	\[\mathbb{E}\left(\sum_{u\in\mathcal{M}_t}F(X_u(s), 0\leq s\leq t)\right) = \mathbb{E}(|\mathcal{M}_t|)\mathbb{E}(F(B(s),0\leq s\leq t))\]
\end{lemma}

\begin{proof}[Proof of Lemma \ref{lem3}]
	As $\gamma \geq 0$, it holds that 
	$\mathbb{E}_{(0)}(\gamma) = \int_0^{\infty} \mathbb{P}_{(0)}(\gamma \geq x) dx$. Then, in order to prove that $\mathbb{E}_{(0)}(\gamma) $ is finite it suffices to show that the function $f: \mathbb{R}^{+}\rightarrow [0,1]$ given by $f(x)= \mathbb{P}_{(0)}(\gamma \geq x)$ is integrable. 
	For $x>0$, we have 
	\begin{equation}\label{eq: lem1}
		\begin{split}
			\mathbb{P}_{(0)}\big(\gamma \geq x\big) &= \mathbb{P}_{(0)}\big(\gamma \geq x, T > \sqrt{x}\big) + \mathbb{P}_{(0)}\big(\gamma \geq x, T \leq \sqrt{x}\big)\\
			& \leq \mathbb{P}_{(0)}\big(T > \sqrt{x}\big) + \mathbb{P}_{(0)}\Bigg(\exists t\in[0,\sqrt{x}]:\Big|\max_{u\in \mathcal{N}^c_t} X^c_u(t)\Big| \geq x\Bigg).
		\end{split}
	\end{equation}
	
	We claim that $T$ has finite $k$-moment for $k\in\mathbb{N}$. Indeed, by the Markov property of the process $(N_t^c)$, we can prove that there exists a constant $K \in (0,1)$ (that depends only on $c$) such that $\mathbb{P}_{(0)}(T \geq 2n) \leq \left(1-K\right)^n$ and from here, for any $k \in \mathbb{N}$,
	\begin{equation*}
		\mathbb{E}\big(T^k\big) \leq 	 k \sum_{n=1}^{\infty} \int_{0}^{\infty}(2n)^{k-1} \mathbb{P}(T \geq 2(n-1) ) \mathbb{1}_{\{x\in[2(n-1),2n) \}}dx = k \sum_{n=1}^{\infty} (2n)^{k-1} \left(1-K\right)^{n-1} <\infty
	\end{equation*}
	and therefore the function $g(x) = \mathbb{P}(T > \sqrt{x})$  is integrable. By the coupling defined in Subsection \ref{subsec: coupling}, the following inequality holds for the second term of the sum in the last line of \eqref{eq: lem1}:
	\begin{equation}\label{eq: C.2}
		\mathbb{P}_{(0)}\Big(\exists t\in[0,\sqrt{x}]:\Big|\max_{u\in \mathcal{N}^c_t} X^c_u(t)\Big| \geq x\Big) \leq \mathbb{P}_{(0)}\Big(\exists t\in[0,\sqrt{x}]:\Big|\max_{u\in \mathcal{M}_t} X_u(t)\Big| \geq x\Big). 
	\end{equation} 
	We have that
	\[
	\Big\{\exists t\in[0,\sqrt{x}]:|\max_{u\in \mathcal{M}_t} X_u(t)| \geq x\Big\} \subset \Big\{\exists t\in[0,\sqrt{x}]:\max_{u\in \mathcal{M}_t} X_u(t) \geq x \ \text{or} \ \max_{u\in \mathcal{M}_t} X_u(t) \leq -x \Big\},
	\]
	and we know that if $\max_{u\in\mathcal{M}_t} X_u \leq -x$ then $\min_{u\in\mathcal{M}_t} X_u \leq -x$.  Therefore $-\min_{u\in\mathcal{M}_t} X_u \geq x$, so
	\begin{equation}\label{eq: inclusion}
		\begin{split}
			\Big\{\exists t\in[0,\sqrt{x}]: \max_{u\in \mathcal{M}_t} X_u(t) \leq -x \Big\} & \subset
			\Big\{\exists t\in[0,\sqrt{x}]: -\min_{u\in \mathcal{M}_t} X_u(t) \geq x \Big\} \\
			&= \Big\{\exists t\in[0,\sqrt{x}]: \max_{u\in \mathcal{M}_t} (-X_u(t)) \geq x \Big\}
		\end{split} \quad.
	\end{equation} 
	Given that the ancestor of particle $u \in \mathcal{M}_t$ is at the origin at time $0$, we can use \eqref{eq: inclusion} and the symmetric property of standard Brownian motion to obtain that
	\begin{equation}
		\begin{split}
			\mathbb{P}\Big(\exists t\in[0,\sqrt{x}]:\Big|\max_{u\in \mathcal{M}_t} X_u(t)\Big| \geq x\Big)&\leq \mathbb{P}\Big(\exists t\in[0,\sqrt{x}]:\max_{u\in \mathcal{M}_t} X_u(t) \geq x \ \text{or} \ \max_{u\in \mathcal{M}_t} (-X_u(t)) \geq x \Big)\\
			&\leq 2 \mathbb{P}\Big(\exists t\in[0,\sqrt{x}]:\max_{u\in \mathcal{M}_t} X_u(t) \geq x\Big) .
		\end{split}
	\end{equation} 
	
	We have $\{\exists t\in[0,\sqrt{x}]:\max_{u\in \mathcal{M}_t} X_u(t) \geq x\} = \{\exists t\in[0,\sqrt{x}]: \exists u\in \mathcal{M}_t \ \text{such that} \  X_u(t) \geq x\}$.  Assume at time $t = \sqrt{x}$ the system has $\big| \mathcal{M}_{\sqrt{x}}\big|$ particles then for $t\in[0,\sqrt{x}],  \big| \mathcal{M}_{t}\big| \leq \big| \mathcal{M}_{\sqrt{x}}\big|$. Then we have that  $\{\exists t\in[0,\sqrt{x}]: \exists u\in \mathcal{M}_t \ \text{such that} \  X_u(t) \geq x\}  = \bigcup_{ u\in \mathcal{M}_{\sqrt{x}}} \{\exists t\in[0,\sqrt{x}]:   X_u(t) \geq x\}$. That implies 
	\begin{equation}
		\begin{split}
			2 \mathbb{P}_{(0)}\Big(\exists t\in[0,\sqrt{x}]:\max_{u\in \mathcal{M}_t} X_u(t) \geq x\Big) & = 2 \mathbb{P}_{(0)}\Big(\bigcup_{ u\in \mathcal{M}_{\sqrt{x}}} \{\exists t\in[0,\sqrt{x}]:   X_u(t) \geq x\}\Big)\\
			&=2  \mathbb{E}_{(0)} \Big(\mathbb{1}_{\bigcup_{ u\in \mathcal{M}_{\sqrt{x}}} \{\exists t\in[0,\sqrt{x}]:   X_u(t) \geq x\}}\Big)\\
			& \leq 2 \mathbb{E}_{(0)} \Bigg(\sum_{u\in \mathcal{M}_{\sqrt{x}}}\mathbb{1}_{\{ \exists t\in[0,\sqrt{x}]: X_u(t) \geq x\}}\Bigg)
		\end{split}		 
	\end{equation} 
	We can use now Lemma \ref{lem: manytone} to obtain
	\begin{equation*}
		2 \mathbb{E}_{(0)} \Bigg(\sum_{u\in \mathcal{M}_{\sqrt{x}}}\mathbb{1}_{\{ \exists t\in[0,\sqrt{x}]: X_v(t) \geq x\}}\Bigg) = 2 \mathbb{E}_{(0)}\Big(| \mathcal{M}_{\sqrt{x}}|\Big) \mathbb{P}\Big(\exists t\in[0,\sqrt{x}]: B(t) \geq x\Big)
	\end{equation*}
	where $B(t)$ is a Brownian motion and  $\mathbb{E}(| \mathcal{M}_{\sqrt{x}}|) = \mathbb{E}(N^c_0)e^{\sqrt{x}}$, with $\mathbb{E}(N^c_0)$ finite. We know that $\{\exists t\in[0,\sqrt{x}]: B(t) \geq x\} = \{\max_{t\in[0,\sqrt{x}]} B(t)\geq x \}$ and using the Reflection Principle we obtain that
	\[
	2 \mathbb{E}_{(0)} \Bigg(\sum_{u\in \mathcal{M}_{\sqrt{x}}}\mathbb{1}_{\{ \exists t\in[0,\sqrt{x}]: X_v(t) \geq x\}}\Bigg) = 4\mathbb{E}(N^c_0) e^{\sqrt{x}} \mathbb{P}(B(\sqrt{x})\geq x).
	\]
	We can bound the tail distribution of a Gaussian variable by 
	\begin{equation*}
		4 \mathbb{E}(N^c_0)e^{\sqrt{x}} \mathbb{P}\left(B(\sqrt{x})\geq x\right)
		=  4 \mathbb{E}(N^c_0)e^{\sqrt{x}} \mathbb{P}\left(x^{1/4}\text{N}(0,1)\geq x\right) \leq  \mathbb{E}(N^c_0)\frac{\sqrt{2}e^{\sqrt{x}-x^{3/2}/2}}{\sqrt{\pi}x^{3/4}},
	\end{equation*}
	and we know that for large $x$, the term $x^{3/2}/2$ dominates $\sqrt{x}$, therefore the term $\exp\{\sqrt{x}-x^{3/2}/2\}$ will decay rapidly, at the same time it will decay faster than any polynomial, which means that effectively this term is integrable as a function of $x$, and  we then deduce the desired result.
\end{proof}
\section{Proof of Lemma \ref{lem: empmsreq}}\label{appendix: empmsreq}
\begin{proof}
	We will analyze the derivative
	\[\partial_t\mathbb{E}_{\psi_{\pi}}\left(\widetilde{F}^K(t,x)\right)
	:= \lim_{h \to 0} \frac{\Delta_h \widetilde{F}^K(t,x)}{h}= 
	\lim_{h \to 0} \frac{1}{h}\mathbb{E}_{\psi_{\pi}}\left(\widetilde{F}^K(t+h,x) - \widetilde{F}^K(t,x)\right)  ,\]
	Note that in the interval $[t,t+h]$, we have three main events: nor branching or competition event;  a single branching event; or a single competition event. Any other outcome is of the order $O(h^2)$. Next, we will analyze each one of these main scenarios.
	\begin{itemize}
		\item \textbf{Nor branching or competition event.}  In this case, we have that 
		\begin{equation}\label{6.4}
			\Delta_h \widetilde{F}^K(t,x)\mathbb{1}_{\text{no event}}= \frac{1}{N_t^K}\sum_{u \in \mathcal{N}^K_{t}}\Big(\mathbb{1}_{\{X^K_u(t+h)\leq x\}} - \mathbb{1}_{\{X^K_u(t)\leq x\}}\Big)
		\end{equation}
		and from here, using that $\mathbb{P}_{\psi_{\pi}}\big(N_t^K =N\big)=\pi(N)$, we have that 
		\begin{equation*}\label{eq: Deltahfirstterm}
			\begin{split}   	
				\lim_{h\to 0}\frac{1}{h}\mathbb{E}_{\psi_{\pi}}\Big[\Delta_h \widetilde{F}^K(t,x)\mathbb{1}_{\text{no event}} \Big]
				&= \sum\limits_{N=1}^{\infty}\frac{1}{N}\sum_{i=1}^{N}\lim_{h\to 0}\frac{\mathbb{P}_{\psi_{\pi}}(X_{u_i}(t+h)<x)- \mathbb{P}_{\psi_{\pi}}(X_{u_i}(t)<x)}{h}\pi(N)\\
				& = \sum\limits_{N=1}^{\infty}\frac{1}{N}\sum_{i=1}^{N} \partial_t\mathbb{P}_{\psi_{\pi}}(X_u(t)<x)\pi(N)\\
				&=\frac{1}{2}\partial_{xx}\mathbb{E}_{\psi_{\pi}}\left[\frac{1}{N^K_t}\sum_{u \in \mathcal{N}^K_{t}}\mathbb{1}_{\{X_u(t)<x\}}\right]
			\end{split},
		\end{equation*}
		where the last equality holds  as  each particle $u$ moves as a Brownian motion.
		\item \textbf{A single branching event.} Denote by $B$ the event of a single branching on $(t,t+h]$ and $B_u$ the event where particle $u$ branches. In this case
		\[\begin{split}
			\mathbb{E}_{\psi_{\pi}}\Big[\Delta_h \widetilde{F}^K(t,x)\mathbb{1}_{B}\Big] &= \sum_{N=1}^{\infty}\sum_{i\geq 1}^N\mathbb{E}_{\psi_{\pi}}\Big[\Delta_h \widetilde{F}^K(t,x)\mathbb{1}_{B_{u_i}}\Big]\pi(N) \\
			& = \sum_{N=1}^{\infty}\sum_{i\geq 1}^N\mathbb{E}_{\psi_{\pi}}\Big[\Delta_h \widetilde{F}^K(t,x)\Big|B_{u_i}\Big]\mathbb{P}_{\psi_{\pi}}(B_{u_i})\pi(N)
		\end{split}\]
		where $\mathbb{P}_{\psi_{\pi}}(B_{u_i}) = h + O(h^2)$. The spatial movement is independent of the reproduction event, so
		\begin{equation*}
			\begin{split}
				\sum_{i\geq 1}^N\mathbb{E}_{\psi_{\pi}}\Big[\Delta_h \widetilde{F}^K(t,x)|B_{u_i}\Big]
				& = \sum_{i\geq 1}^N\mathbb{E}_{\psi_{\pi}}\Bigg[\frac{1}{N+1}\sum_{j=1}^{N}\mathbb{1}_{\{X^K_{u_j}(t+h)<x\}} - \frac{1}{N}\sum_{j=1}^{N}\mathbb{1}_{\{X^K_{u_j}(t)<x\}} \\ 
				& +\frac{1}{N+1}\mathbb{1}_{\{X^K_{\widetilde{u}_i}(t+h)<x\}} \Big|B_{u_i}\Bigg] \\
				&=\frac{N}{N+1}\mathbb{E}_{\psi_{\pi}}\Bigg[\sum_{j=1}^{N}\mathbb{1}_{\{X^K_{u_j}(t+h)<x\}}\Bigg] - \mathbb{E}_{\psi_{\pi}}\Bigg[\sum_{j=1}^{N}\mathbb{1}_{\{X^K_{u_j}(t)<x\}}\Bigg] \\ 
				& +\sum_{i\geq 1}^N\mathbb{E}_{\psi_{\pi}}\Bigg[\frac{1}{N+1}\mathbb{1}_{\{X^K_{u_i}(t+h)<x\}} \Big|B_{u_i}\Bigg]
			\end{split},
		\end{equation*}
		where in the last equality we use that $X^K_{\tilde{u}_i}(t+h)\overset{d}{=}X^K_{u_i}(t+h)$ when $\tilde{u}$ is a descendant of individual $u$, born in $(t, t+h]$. Hence, 
		\begin{equation*}
			\begin{split}
				\sum_{i\geq 1}^N\mathbb{E}_{\psi_{\pi}}\Big[\Delta_h \widetilde{F}^K(t,x)|B_{u_i}\Big]
				&=\frac{N}{N+1}\mathbb{E}_{\psi_{\pi}}\Bigg[\sum_{j=1}^{N}\mathbb{1}_{\{X^K_{u_j}(t+h)<x\}}\Bigg] - \mathbb{E}_{\psi_{\pi}}\Bigg[\sum_{j=1}^{N}\mathbb{1}_{\{X^K_{u_j}(t)<x\}}\Bigg] \\ 
				& +\frac{1}{N+1}\mathbb{E}_{\psi_{\pi}}\Bigg[\sum_{i\geq 1}^N\mathbb{1}_{\{X^K_{u_i}(t+h)<x\}}\Bigg]\\
				&=\mathbb{E}_{\psi_{\pi}}\Bigg[\sum_{j=1}^{N}\mathbb{1}_{\{X^K_{u_j}(t+h)<x\}}\Bigg] - \mathbb{E}_{\psi_{\pi}}\Bigg[\sum_{j=1}^{N}\mathbb{1}_{\{X^K_{u_j}(t)<x\}}\Bigg]
			\end{split}.
		\end{equation*}
		Therefore, we have that
		\begin{equation}\label{6.5}
			\begin{split}
				\frac{1}{h}\mathbb{E}_{\psi_{\pi}}\Big[\Delta_h \widetilde{F}^K(t,x)\mathbb{1}_{B} \Big] &= \frac{1}{h}\sum_{N=1}^{\infty}\sum_{i= 1}^N\mathbb{E}_{\psi_{\pi}}\Big[\mathbb{1}_{\{X^K_u(t+h)\leq x\}} - \mathbb{1}_{\{X^K_u(t)\leq x\}}\Big]\Big(h + O(h^2)\Big)\pi(N)
			\end{split}
		\end{equation}
		which vanishes  as $h$ approaches to zero.
		\item \textbf{A single competition event.} Denote by $C$ the event of a single competition on $(t,t+h]$ and $C_{u,v}$ the event where particle $u$ and $v$ competes. In this case, 
		\[\begin{split}
			\mathbb{E}_{\psi_{\pi}}\Big[\Delta_h \widetilde{F}^K(t,x)\mathbb{1}_{C}\Big] &= \sum_{N=1}^{\infty}\sum_{j,w: w\neq j}^N\mathbb{E}_{\psi_{\pi}}\Big[\Delta_h \widetilde{F}^K(t,x)\mathbb{1}_{C_{u_w,u_j}}\Big]\pi(N) \\
			& = \sum_{N=1}^{\infty}\sum_{j,w: w\neq j}^N\mathbb{E}_{\psi_{\pi}}\Big[\Delta_h \widetilde{F}^K(t,x)|C_{u_w,u_j}\Big]\mathbb{P}_{\psi_{\pi}}(C_{u_w,u_j})\pi(N)
		\end{split}\]
		where $\mathbb{P}_{\psi_{\pi}}(C_{u_w,u_j}) = c_Kh +O(h^2)$
		\begin{equation*}
			\begin{aligned}
				\sum_{j,w: w\neq j}^N\mathbb{E}_{\psi_{\pi}}\Big[\Delta_h \widetilde{F}^K(t,x)&|C_{u_w,u_j}\Big]
				=	\sum_{j,w: w\neq j}^N\mathbb{E}_{\psi_{\pi}}\Bigg[\frac{1}{N-1}\sum_{i\geq 1 }^N\mathbb{1}_{\{X^K_{u_i}(t+h)<x\}}- \frac{1}{N}\sum_{i\geq 1}^N\mathbb{1}_{\{X^K_{u_i}(t)<x\}}  \\
				&- \frac{1}{N-1}\mathbb{1}_{\{X_{ k_{u_w,u_j}(\tau)}(  t+h)<x\}}\Big|C_{u_w,u_j}\Big]\\
				& =	N(N-1)\mathbb{E}_{\psi_{\pi}}\Bigg[\frac{1}{N-1}\sum_{i\geq 1 }^N\mathbb{1}_{\{X^K_{u_i}(t+h)<x\}}\Bigg]\\
				& - N(N-1)\mathbb{E}_{\psi_{\pi}}\Bigg[\frac{1}{N}\sum_{i\geq 1}^N\mathbb{1}_{\{X^K_{u_i}(t)<x\}}\Bigg]  \\
				&- \sum_{j,w: w\neq j}^N\frac{1}{N-1}\mathbb{E}_{\psi_{\pi}}\Big[\mathbb{1}_{\{X_{ k_{u_w,u_j}(\tau)}(  t+h)<x\}}\Big]\\
				& =	\mathbb{E}_{\psi_{\pi}}\Bigg[N\sum_{i\geq 1 }^N\mathbb{1}_{\{X^K_{u_i}(t+h)<x\}} \Big]- (N-1)\mathbb{E}_{\psi_{\pi}}\Bigg[\sum_{i\geq 1}^N\mathbb{1}_{\{X^K_{u_i}(t)<x\}}\Bigg]  \\
				&- \frac{1}{N-1}\mathbb{E}_{\psi_{\pi}}\Big[\sum_{j,w: w\neq j}^N\mathbb{1}_{\{X_{ k_{u_w,u_j}(\tau)}(  t+h)<x\}}\Big]
			\end{aligned}
		\end{equation*}
		\noindent
		where $t<\tau\leq t+h$ is the time of the first competition event and $k_{u_w,u_j}(\tau)$ is the label of the particles $u_w$ or $u_j$ with the smaller position at time $\tau$. Given that at this time the particle with label $k_{u_w,u_j}(\tau)$ is removed from the system, we consider $X^K_{k_{u_w,u_j}(\tau)}(t+h)$ to be the position of a copy of particle $k_{u_w,u_j}(\tau)$ that kept moving as a Brownian motion from $\tau$ to $t+h$.  For this, we have 
		\begin{equation}\label{6.6}
			\sum_{j,w: w\neq j}^N\mathbb{1}_{\{X_{k_{v_w,v_j}(\tau)}(t+h)<x\}} =   2\sum_{i,j: i\neq j}^N\mathbb{1}_{\{X^K_{u_i}(t+h)<x\}}\mathbb{1}_{\{X^K_{u_i}(\tau)<X^K_{v_j}(\tau)\}}.
		\end{equation}
		Therefore
		\[\begin{split}
			\frac{1}{h}\mathbb{E}_{\psi_{\pi}}\Big[\Delta_h \widetilde{F}^K(t,x)&\mathbb{1}_{C}\Big] = \frac{1}{h}\sum_{N=1}^{\infty}\mathbb{E}_{\psi_{\pi}}\Big[N\sum_{i\geq 1 }^N\mathbb{1}_{\{X^K_{u_i}(t+h)<x\}} - (N-1)\sum_{i\geq 1 }^N\mathbb{1}_{\{X^K_{u_i}(t)<x\}}\\
			&- \frac{2}{N-1}\sum_{i,j: i\neq j}^N\mathbb{1}_{\{X^K_{u_i}(t+h)<x\}}\mathbb{1}_{\{X^K_{u_i}(\tau)<X^K_{v_j}(\tau)\}}\Big]\Big(c_Kh +O(h^2)\Big)\pi(N)
		\end{split}\]
		so we can conclude
		\[\begin{split}
			\lim_{h\to 0}\frac{1}{h}\mathbb{E}_{\psi_{\pi}}\Big[\Delta_h \widetilde{F}^K(t,x)\mathbb{1}_{C}\Big] = \mathbb{E}_{\psi_{\pi}}\Big[&c_K\sum_{u\in\mathcal{N}^K_t}\mathbb{1}_{\{X^K_{u}(t)<x\}}\\
			&- \frac{2c_K}{N_t^K-1}\sum_{u,v\in \mathcal{N}_t^K}\mathbb{1}_{\{X^K_{u}(t)<x\}}\mathbb{1}_{\{X^K_{u}(t)<X^K_{v}(t)\}}\Big].
		\end{split}\]
	\end{itemize}
	Based in the analysis above, we obtain that 
	\begin{equation*}
		\begin{split}
			\partial_t\mathbb{E}_{\psi_{\pi}}\left(\widetilde{F}^K(t,x)\right) & = \frac{1}{2}\partial_{xx}\mathbb{E}_{\psi_{\pi}}\left(\widetilde{F}^K(t,x)\right) \\    
			& \mathbb{E}_{\psi_{\pi}}\Big[c_K\sum_{u\in\mathcal{N}^K_t}\mathbb{1}_{\{X^K_{u}(t)<x\}}- \frac{2c_K}{N_t^K-1}\sum_{u,v\in \mathcal{N}_t^K}\mathbb{1}_{\{X^K_{u}(t)<x\}}\mathbb{1}_{\{X^K_{u}(t)<X^K_{v}(t)\}}\Big]          \end{split} ,
	\end{equation*}
	and finally we can give an expression for the derivative in the integrand on \eqref{eq: dervempmean} in terms of $\widetilde{F}^K(t,x)$, namely 
	\begin{equation}\label{fini_eq}
		\partial_t\mathbb{E}_{\psi_{\pi}}(\widetilde{F}^K(t,x)) = \frac{1}{2}\partial_{xx} \mathbb{E}_{\psi_{\pi}}\big(\widetilde{F}^K(t,x))	-\mathbb{E}_{\psi_{\pi}}\Big[\frac{c_K (N_t^K)^2}{N_t^K-1}(\widetilde{F}^K(t,x) 
		- \widetilde{F}^K(t,x)^2\big)\Big],
	\end{equation}
	where we use \eqref{6.6} and the fact that
	\[ \sum_{\stackrel{u,v\in\mathcal{N}^K_t}{u\neq v}}\mathbb{1}_{\{\min(X^K_u(t),X^K_v(t))<x\}}  = 
	2(N^K_t-1)\sum_{u\in\mathcal{N}^K_t}\mathbb{1}_{\{X^K_u(t) <x\}} - \left(\sum_{u\in\mathcal{N}^K_t}\mathbb{1}_{\{X^K_u(t) <x\}}\right)^2+ \sum_{u\in\mathcal{N}^K_t}\mathbb{1}_{\{X^K_u(t) <x\}}.\]
	
\end{proof} 

		\bibliographystyle{amsplain}
		

\end{document}